\documentclass[12pt]{iopart}

\usepackage{amsthm,amssymb,bbm,graphicx,psfrag,xcolor,enumerate,bm,url}
\usepackage[latin1]{inputenc}

\newtheorem{theorem}{Theorem}[section]

\newtheorem{remark}[theorem]{Remark}
\newtheorem{lemma}[theorem]{Lemma}
\newtheorem{definition}[theorem]{Definition}

\def\req#1{{\rm(\ref{eq:#1})}}
\newcommand{\labeq}[1]{\label{eq:#1}}
\def\norm#1{\hspace{0.2ex} \|#1\| \hspace{0.2ex}}

\newcommand{\R}{\ensuremath{\mathbbm{R}}} 
\newcommand{\N}{\ensuremath{\mathbbm{N}}} 

\newcommand{\range}{\mathcal R}
\newcommand{\kernel}{\mathcal N}
\newcommand{\dx}[1][x]{\ensuremath{\ {\rm{d}} #1}}
\newcommand{\supp}{\ensuremath{\mathrm{supp}}}

\newcommand{\kommentar}[1]{}

\begin{document}

\title[Interpolation of missing electrode data in EIT]{Interpolation of missing electrode data in electrical impedance tomography\footnotemark}

\renewcommand{\footnoterule}{%
  \kern -3pt
  \hrule width \textwidth height 1pt
  \kern 2pt
}
\footnotetext{%
\scriptsize
This is an author-created, un-copyedited version of an article published in \emph{Inverse Problems} \textbf{31}(11), 115008, 2015.
IOP Publishing Ltd is not responsible for any errors or omissions in this version of the manuscript or any version derived from it. The Version of Record is available online at 
\url{http://dx.doi.org/10.1088/0266-5611/31/11/115008}.
}     
     
\author{Bastian Harrach$^{1}$}
\address{$^{1}$ Institute of Mathematics, Goethe University Frankfurt, Germany}
\ead{\mailto{harrach@math.uni-frankfurt.de}}

\begin{abstract}
Novel reconstruction methods for electrical impedance tomography (EIT) often require voltage measurements on current-driven electrodes.
Such measurements are notoriously difficult to obtain in practice as they tend to be affected by unknown contact impedances
and require problematic simultaneous measurements of voltage and current.

In this work, we develop an interpolation method that predicts the voltages on current-driven electrodes from 
the more reliable measurements on current-free electrodes for difference EIT settings, where a conductivity change
is to be recovered from difference measurements.
Our new method requires the a-priori knowledge of an upper bound of the conductivity change, and utilizes
this bound to interpolate in a way that is consistent with the special geometry-specific smoothness of difference EIT data.

Our new interpolation method is computationally cheap enough to allow for real-time applications, and 
simple to implement as it can be formulated with the standard sensitivity matrix. We numerically evaluate
the accuracy of the interpolated data and demonstrate the feasibility of using interpolated measurements for a 
monotonicity-based reconstruction method.
\end{abstract}


\section{Introduction}
\label{Sec:intro}

Electrical impedance tomography (EIT) is a novel technique that 
images the con\-duc\-tiv\-i\-ty distribution inside a subject from electric voltage and current
measurements on the subject's boundary. EIT has promising applications in several fields
including medical imaging, geophysics, nondestructive material testing and
monitoring of industrial processes. For a broad overview on the developments in EIT in the last decades let us refer to 
\cite{henderson1978impedance,barber1984applied,wexler1985impedance,newell1988electric,metherall1996three,cheney1999electrical,borcea2002electrical,borcea2003addendum,lionheart2003eit,holder2004electrical,bayford2006bioimpedance,uhlmann2009electrical,adler2011electrical,martinsen2011bioimpedance,seo2013electrical}, and the references therein.

The inverse problem of reconstructing the conductivity from voltage-current-measurements is known to be highly ill-posed and non-linear,
and the reconstructions suffer from an enormous sensitivity to modeling and measurement errors.
To alleviate this problem, most applications concentrate on reconstructing a conductivity change from the difference of two
measurements, which is less affected by modeling errors, cf.\ the above cited works for time-difference EIT
and \cite{seo2008frequency,harrach2009detecting,harrach2010factorization} for weighted frequency-difference EIT.
In some applications one may further restrict the problem to recovering only the outer support of the conductivity change (the so-called anomaly or inclusion detection problem), which is unaffected by linearization errors, cf.\ \cite{harrach2010exact}.

In practice, the inverse problem of difference EIT is usually linearized, and the
conductivity change is then recovered from minimizing a regularized version of the linearized residuum functional, see subsection~\ref{subsect:linearized}. This approach is well-established and very flexible as it can incorporate all available measurements
in the residuum functional. However, there is almost no theoretical justification for this approach.
Standard convergence results for linearization-based iterative solvers for non-linear inverse problems 
are based on certain assumptions, the so-called source conditions and the tangential cone condition. 
It is not clear, whether realistic conductivity distributions fulfill the required abstract source conditions, 
and the validity of the tangential cone condition is a long-standing open problem in EIT, cf.\ \cite{lechleiter2008newton} for 
a recent contribution.

On the other hand, rigorously justified reconstruction methods for EIT have been developed in the more mathematically-oriented 
community. An explicit reconstruction formula that is based on the global uniqueness proof of Nachman \cite{nachman1996global}
is known as the d-bar method, cf., e.g., \cite{siltanen2000implementation,knudsen2002inverse,knudsen2003new,knudsen2004numerical,
isaacson2004reconstructions,knudsen2005reconstruction,isaacson2006imaging,knudsen2007d,Knudsen2009}.
Rigorously justified inclusion detection methods for EIT include the enclosure method
(see \cite{ikehata1999draw,Bru00,ikehata2000reconstruction,ikehata2000numerical,ikehata2002regularized,ikehata2004electrical,ide2007probing,uhlmann2008reconstructing,ide2010local}),
 the Factorization Method (see
\cite{Han03,Geb05,kirsch2005factorization,azzouz2007factorization,Han07,Hyv07,lechleiter2008factorization,Nac07_2,GH07,gebauer2008localized,hakula2009computation,harrach2009detecting,Schmitt2009,harrach2010factorization,schmitt2011factorization,haddar2013numerical,chaulet2014factorization,choi2014regularizing,barthdetecting} and the recent overviews
\cite{Kirsch_book07,hanke2011sampling,harrach2013recent}),
and the recently emerging Monotonicity Method \cite{Tamburrino02,Tamburrino06,harrach2013monotonicity,Zhou2015Monotonicity_preprint}.

Rigorously justified reconstruction methods rely on interpreting the measurements as (an approximation to) the continuous
Neumann-to-Dirichlet (NtD) operator of the underlying partial differential equation, and the operator structure plays a major role in the methods' theoretical foundation. In practical EIT applications, the continuous NtD-operator has to be replaced by the 
discrete matrix of voltage-current measurements on a finite number of electrodes.
Still, in that case certain rigorous properties can be guaranteed, cf. \cite{harrach2015resolution},
and rigorously justified methods have successfully been 
applied to real data situations, cf., e.g. \cite{azzouz2007factorization,harrach2010factorization,Zhou2015Monotonicity_preprint}.

Practical applications of rigorously justified methods require, however, a full matrix of measurements 
that includes voltages on current-driven electrodes. Such measurements are notoriously difficult to obtain in practice as they tend to be affected by unknown contact impedances and require problematic simultaneous measurements of voltage and current.
Accordingly, there is a long-standing disagreement in the engineering community whether practical EIT systems should take voltage measurents on current-driven electrodes. Some groups have successfully developed systems that utilize these measurements, cf. Rensselaer's ACT 4 system \cite{saulnier2007electrical} or the system of the Dartmouth group
\cite{halter2004design,halter2008broadband}.
However, in several other systems, including the recently commercially launched PulmoVista\textsuperscript{\textregistered} 500 from Dr\"ager Medical, voltage measurements on current-driven electrodes are not taken.

In this work, we develop an interpolation method that predicts the voltages on current-driven electrodes from 
the more reliable measurements on current-free electrodes for difference EIT settings, where a conductivity change
is to be recovered from difference measurements. 
Our new method is based on utilizing the special geometry-specific smoothness of difference EIT data in the following way.
The difference of two electric potentials for different conductivities solves an elliptic PDE with homogeneous boundary data and a source term generated by the conductivity change. The smoothness of the measured boundary voltage differences will depend on how far the conductivity change is supported away from the boundary. We assume the a-priori knowledge of an upper bound $B$ of the conductivity change. With this bound we calculate minimal source terms on $B$ that are consistent with the measurements on the current-free electrodes. The minimal measurement-consistent source terms are then used to calculate the missing voltages on the current-driven electrodes.

We will show that this new interpolation procedure uniquely determines the missing voltages on current-driven electrodes, and derive an analytic formula to calculate the interpolated voltages. Notably, 
our new interpolation method is computationally cheap enough to allow for real-time applications, and 
simple to implement as it can be formulated with the standard sensitivity matrix. For a setting with $m$ electrodes,
the interpolated voltages can be obtained by solving a linear system with a $m\times m$-matrix that is constructed from the columns of the sensitivity matrix, see theorem~\ref{thm:main} and remark~\ref{remark:main} for the details.

The paper is organized as follows. In section \ref{Sec:missing}, we describe a typical difference EIT setting with
adjacent-adjacent current driving patterns, and summarize the standard linearized reconstruction method and the novel
matrix-based monotonicity method. Section \ref{section:main} contains our main result on how to interpolate
the voltages on current-driven electrodes using geometry-specific smoothness of difference EIT data.
In section \ref{sect:numerics}, we illustrate our new interpolation method
and numerically evaluate the interpolation error.
We also demonstrate the feasibility of using interpolated measurements for a monotonicity-based reconstruction method
for two- and three-dimensional numerical examples.

\section{Missing electrode data in electrical impedance tomography}
\label{Sec:missing}

\subsection{The setting}
\label{Subsec:setting}

Let $\Omega\subset \R^n$, $n \geq 2$
be a bounded domain describing the imaging domain and let $\sigma\in L^\infty_+(\Omega)$ be its conductivity distribution.
$\Omega$ is assumed to have piecewise smooth boundary $\partial \Omega$ with outer normal
vector $\nu$. $L_+^\infty$ denotes the subspace of $L^\infty$-functions with positive essential infima.

For electrical impedance tomography (EIT), several electrodes $\mathcal E_l\subset \partial \Omega$, $l=1,\ldots,m$, are attached to the imaging domain's surface. We assume that the electrodes $E_l$ are relatively open and connected subsets of $\partial \Omega$, that they are 
perfectly conducting and that contact impedances are negligible (the so-called \emph{shunt model}, cf., e.g., \cite{cheney1999electrical}). 
When a current $I_l\in \R$ is driven through the $l$-th electrode, with $\sum_{l=1}^m I_l=0$, the electric potential inside the imaging domain is given by the solution $u_{\sigma}\in H^1(\Omega)$ of
\begin{eqnarray}
\labeq{shunt1} \nabla \cdot (\sigma \nabla u_{\sigma}) = 0 & \quad  \mbox{ in } \Omega,\\
\labeq{shunt2} \int_{\mathcal E_l} \sigma \partial_\nu u_{\sigma} \dx[s] = I_l & \quad \mbox{ for } l=1,\ldots,m,\\
\labeq{shunt3} \sigma \partial_\nu u_{\sigma} =0 & \quad  \mbox{ on } \partial \Omega\setminus \bigcup_{l=1}^m \mathcal E_l,\\
\labeq{shunt4} u_{\sigma}|_{\mathcal E_l}=\mathrm{const.} & \quad \mbox{ for all } l=1,\ldots,m.
\end{eqnarray}
$u_{\sigma}$ is uniquely determined up to the addition of constant functions.


In a standard adjacent-adjacent driving configuration the voltage-current-mea\-sure\-ments are carried out in the following way. For each $k=1,\ldots,m$, we drive a current of $I_k=1$ and $I_{k+1}=-1$ through the $k$-th pair of electrodes $(\mathcal E_k, \mathcal E_{k+1})$ while all other electrodes are kept insulated. Here, and throughout the paper, the 
electrode index is always considered modulo $m$, i.e. the index $m+1$ refers to the first electrode, and the index $0$ refers to the $m$-th electrode.

The resulting electric potential $u^{(k)}$ solves \req{shunt1}--\req{shunt4} with 
\[
I_l:=\delta_{k,l} - \delta_{k+1,l}\quad \mbox{ for $l=1,\ldots,m$.}
\]
While keeping up the electric current between the $k$-th electrode pair, we measure the volt\-age between
the $j$-th pair of electrodes, i.e., 
\begin{equation}\labeq{Def_Ujk}
U_{jk}(\sigma):=u^{(k)}_\sigma|_{\mathcal E_j}-u^{(k)}_\sigma|_{\mathcal E_{j+1}}.
\end{equation}

Repeating this process for all $j,k=1,\ldots,m$ we obtain the measurement matrix
\[
\fl
U(\sigma)=\left(
\begin{array}{c c c c c c}
\colorbox{gray}{$U_{11}(\sigma)$} & \colorbox{gray}{$U_{12}(\sigma)$} & U_{13}(\sigma) & \ldots & U_{1,m-1}(\sigma) & \colorbox{gray}{$U_{1,m}(\sigma)$}\\
\colorbox{gray}{$U_{21}(\sigma)$} & \colorbox{gray}{$U_{22}(\sigma)$} & \colorbox{gray}{$U_{23}(\sigma)$} & \ldots & U_{2,m-1}(\sigma) & U_{2,m}(\sigma)\\
U_{31}(\sigma) & \colorbox{gray}{$U_{32}(\sigma)$} & \colorbox{gray}{$U_{33}(\sigma)$} & \ddots &  & U_{3,m}(\sigma)\\
\vdots & \vdots & \ddots & \ddots & \ddots & \vdots\\
U_{m-2,1}(\sigma) & U_{m-2,2}(\sigma) & & \ddots & \colorbox{gray}{$U_{m-2,m-1}(\sigma)$} & \colorbox{gray}{$U_{m-2,m}(\sigma)$}\\
U_{m-1,1}(\sigma) & U_{m-1,2}(\sigma) & U_{m-1,3}(\sigma) & \ldots & \colorbox{gray}{$U_{m-1,m-1}(\sigma)$} & \colorbox{gray}{$U_{m-1,m}(\sigma)$}\\
\colorbox{gray}{$U_{m,1}(\sigma)$} & U_{m,2}(\sigma) & U_{m,3}(\sigma) & \ldots & \colorbox{gray}{$U_{m,m-1}(\sigma)$} & \colorbox{gray}{$U_{m,m}(\sigma)$}
\end{array}
\right)
\]

The gray marked entries $U_{jk}$ with $|j-k|\leq 1$ (modulo $m$) correspond to voltage mea\-sure\-ments on current-driven electrodes. In practice, these measurements are usually considered erroneous since they tend to be affected by contact impedances that are not considered in the model and require problematic simulateneous measurement of voltage and current. 

Note that the measurement matrix contains some redundancy. Using \req{shunt1}--\req{shunt4} one can show that 
\[
U_{jk}(\sigma)=\int_\Omega \sigma  \nabla u_{\sigma}^{(j)} \cdot \nabla u_{\sigma}^{(k)}\dx = U_{kj}(\sigma),
\]
so that $U(\sigma)\in \R^{m\times m}$ is symmetric. Moreover, it immediately follows from \req{Def_Ujk} that the 
entries of each column in $U(\sigma)$ sum to zero. Hence, without voltage measurements on current-driven electrodes, 
essentially $m$ entries are missing in $U(\sigma)$. 



\subsection{Linearized reconstruction methods for difference EIT}\label{subsect:linearized}

In difference EIT, one compares two voltage measurements $U(\sigma)$, $U(\sigma_0)$ 
to reconstruct the conductivity change $\sigma-\sigma_0$ with respect to some reference conductivity distribution $\sigma_0$. In a standard linearized approach, one obtains an approximation $\kappa$ to the conductivity change by solving (a regularized version of) the linearized equation
\begin{equation}\labeq{EIT_linearized_Frechet}
U'(\sigma_0)\kappa = V
\end{equation}
where $V:=U(\sigma)-U(\sigma_0)$ is the difference of the measured data, and $U'(\sigma_0):\ L^\infty(\Omega)\to \R^{m\times m}$ is the Fr\'echet derivative of the voltage measurements,
\[
U'(\sigma_0):\ \kappa \mapsto \left( -\int_{\Omega} \kappa \nabla u_{\sigma_0}^{(j)} \cdot \nabla u_{\sigma_0}^{(k)} \dx \right)_{j,k=1,\ldots,m}\in \R^{m\times m},
\]
where $u_{\sigma_0}^{(j)}$ solves \req{shunt1}--\req{shunt4} with reference conductivity $\sigma_0$ and electric current driven through the $j$-th and $(j+1)$-th electrode.

We discretize the imaging domain $\overline{\Omega}=\bigcup_{i=1}^r P_i$ into $r$ disjoint pixels $P_i$ and assume that $\kappa$ is piecewise constant with respect to that partition, i.e.,
\[
\kappa(x)=\sum_{i=1}^r \kappa_i \chi_{P_i}(x).
\]
 Then equation \req{EIT_linearized_Frechet} becomes
\begin{equation}\labeq{EIT_linearized_S}
S \kappa = \boldsymbol{V},
\end{equation}
where $\kappa\in \R^r$ is a column vector containing the entries $\kappa_i$, $i=1,\ldots,r$,
and $\boldsymbol{V}\in \R^{m^2}$ is obtained by writing the matrix $V\in \R^{m\times m}$ as a long column vector,
i.e.,
\[
\boldsymbol{V}_{(j-1)m+k}=V_{jk}, \quad j,k=1,\ldots,m.
\]
$S\in \R^{m^2\times r}$ is the so-called sensitivity matrix, its entries are given by
\[
S_{(j-1)m+k, i}:=-\int_{P_i} \nabla u_{\sigma_0}^{(j)} \cdot \nabla u_{\sigma_0}^{(k)} \dx,
\quad j,k=1,\ldots,m.
\]

In this simple, yet well-established approach, the matrix structure of the mea\-sure\-ments is completely ignored. As a consequence, erroneous voltage measurements on current-driven electrodes
can simply be deleted from the linear system \req{EIT_linearized_S} by removing both, the corresponding lines in the sensitivity matrix, and the problematic entries in the right hand side of \req{EIT_linearized_S}. Thus, one replaces \req{EIT_linearized_S} by the reduced system
\begin{equation}\labeq{EIT_linearized_S_reduced}
\mathbbm{S} \kappa = \mathbbm{V},
\end{equation}
where $\mathbbm{V}\in \R^{m(m-3)}$
denotes the difference data vector without the voltage measurements on current-driven electrodes, and $\mathbbm{S}\in \R^{m(m-3)\times r}$ denotes the reduced sensitivity matrix, in which the lines corresponding to the problematic measurements are removed.

\subsection{Matrix-based methods in EIT}\label{subsec:matrix_methods}

Several new rigorously justified reconstruction approaches in EIT are based on utilizing the matrix structure of EIT measurements, cf.\ the overview in the introduction. In this subsection we summarize a recent linearized monotonicity-based method \cite{harrach2013monotonicity},
that we will also use to numerically test our interpolation methods developed in the next section. 

For the monotonicity method (but not for the interpolation method developed in the next section),
we will restrict ourself to the so-called definite case that 
the conductivity change is either everywhere positive or everywhere negative. 
Without this assumption a more complicated variant of the method would be required, cf.\ \cite{harrach2013monotonicity}.

Let $S_i$ be the $i$-th column of the sensitivity matrix written as a $m\times m$-matrix, i.e.,
\begin{equation}\labeq{sensitivity_element_wise}
S_i:=-\left( \begin{array}{c c c} \int_{P_i} \nabla u_{\sigma_0}^{(1)} \cdot \nabla u_{\sigma_0}^{(1)} \dx &
\dots & \int_{P_i} \nabla u_{\sigma_0}^{(1)} \cdot \nabla u_{\sigma_0}^{(m)} \dx\\
\vdots & & \vdots\\
\int_{P_i} \nabla u_{\sigma_0}^{(m)} \cdot \nabla u_{\sigma_0}^{(1)} \dx &
\dots & \int_{P_i} \nabla u_{\sigma_0}^{(m)} \cdot \nabla u_{\sigma_0}^{(m)} \dx
\end{array}\right).
\end{equation}

When regarded as matrices and partially ordered with respect to matrix
def\-i\-nite\-ness, the EIT measurements depend on the conductivity in a monotonous way.
For all vectors $g=(g_k)_{k=1}^m\in \R^m$,
\begin{eqnarray}
\int_\Omega \frac{\sigma_0}{\sigma}(\sigma_0-\sigma) \left| \sum_{j=1}^m g_j \nabla u_{\sigma_0}^{(j)}\right|^2 \dx
\nonumber \geq g^T \left( U(\sigma)-U(\sigma_0) \right) g= g^T V g\\ 
\labeq{monotonicity} \geq \int_\Omega (\sigma_0-\sigma) \left| \sum_{j=1}^m g_j \nabla u_{\sigma_0}^{(j)} \right|^2 \dx.
\end{eqnarray}
For the shunt model considered herein, the monotonicity relation \req{monotonicity} is proven in \cite[lemma~3.1]{harrach2015combining}, see also \cite{kang1997inverse,ikehata1998size} for the origin of this inequality. Monotonicity relations similar to \req{monotonicity} have also been used
to simplify the theoretical basis of the Factorization Method \cite{harrach2013recent,arnold2013unique},
to prove theoretical uniqueness results, cf. \cite{gebauer2008localized,harrach2009uniqueness,harrach2012simultaneous,harrachlocal}, and even to derive new hybrid tomography methods \cite{harrach2015combining}.

It follows from \req{monotonicity} that, for all $\beta\geq 0$,
\begin{eqnarray}
\labeq{element_monotonicity1} \beta \chi_{P_i} \leq \sigma_0 - \sigma \quad & \mbox{ implies } \quad  V\geq -\beta S_i, \quad \mbox{ and }\\
\labeq{element_monotonicity2} \beta \chi_{P_i} \leq \frac{\sigma_0}{\sigma}(\sigma - \sigma_0) 
\quad & \mbox{ implies } \quad 
V\leq \beta S_i,
\end{eqnarray}
where the inequalities on the left hand sides of \req{element_monotonicity1}--\req{element_monotonicity2} are to be understood point-wise almost everywhere, and the inequalities on the right hand sides have to be interpreted in terms of matrix definiteness, i.e. $V\geq -\beta S_i$ 
means that the matrix $V+\beta S_i\in \R^{m\times m}$ possesses only non-negative eigenvalues.

The monotonicity inequality \req{monotonicity} also yields that 
\[
|V|=V\geq 0 \mbox{ for $\sigma\leq \sigma_0$,} \quad \mbox{ and } \quad -|V|=V\leq 0 \mbox{ for $\sigma_0\leq \sigma$,}
\]
where $|V|$ denotes the matrix absolute value. For each pixel $P_i$, we define
\begin{equation}\labeq{beta_montonicity}
\beta_i:=\max \{\beta\geq 0:\  \beta S_i\geq -|V| \}.
\end{equation}
If the conductivity $\sigma$ is either everywhere lower than the reference conductivity $\sigma_0$ (i.e., $\sigma\leq \sigma_0$),
or everywhere larger (i.e., $\sigma\geq \sigma_0$), then \req{element_monotonicity1}--\req{element_monotonicity2} yield that for each pixel,
\[
\inf_{x\in P_i} |\sigma(x)- \sigma_0(x)|>0 \quad \mbox{ implies } \quad \beta_i>0.
\]
Hence, the support of the function
\begin{equation}\labeq{monotonicity_indicator}
x\mapsto \sum_{i=1}^r \beta_i \chi_{P_i}(x) 
\end{equation}
will contain the support of the conductivity change $|\sigma-\sigma_0|$ up to the pixel partition.

Moreover, one can show that, in the limit of infinitely many electrodes, when the measurements are given by the Neumann-Dirichlet-operators, $\beta_i=0$ for each pixel that lies outside the outer support of $|\sigma_0-\sigma|$, see \cite{harrach2013monotonicity}. 
Hence, in the limit of noiseless data on infinitely many electrodes, a plot of \req{monotonicity_indicator} will show where the conductivity $\sigma$
differs from its reference state $\sigma_0$. Let us stress that (even though linearized monotonicity tests are used), the non-linear shape reconstruction problem is being solved, cf.\ \cite{harrach2013monotonicity}. 
This is in accordance with the fact that shape information is in some sense unaffected by linearization errors, cf.\ \cite{harrach2010exact}.

Montonicity-based methods require the full matrix of measurements. The def\-i\-nite\-ness tests in 
\req{beta_montonicity} do not seem possible without knowledge of the main diagonals of the 
measurement matrices, which contain the voltages on current-driven electrodes. In the next section
we will develop a method that uses the geometry-specific smoothness of difference measurements in order
to interpolate the missing voltages from the measurements on current-free electrodes.

\section{Interpolation of missing electrode data in EIT}
\label{section:main}
Though voltage measurements on current driven electrodes are usually considered erroneous, standard EIT systems such as the 32-channel Swisstom pioneer EIT system nevertheless allow the user to access these measurements. Hence, the simplest way to implement a method that relies on the measurements matrix structure is to simply ignore the problem and use the full measurement matrix as given by the EIT system. 
This approach has been taken for phantom experiment data in \cite{Zhou2015Monotonicity_preprint}, 
and showed that the use of monotonicity-based constraints improves the reconstructions compared to a standard linearized approach \req{EIT_linearized_S_reduced}. In the following, we 
will study how to interpolate the voltages on current-driven electrodes from the measurements on current-free electrodes.

\subsection{Motivation for interpolating missing data}\label{subsect:motivation_interpolation}

In many applications, small electrodes are used, the conductivity is assumed to be homogeneous
in a neighbourhood of the boundary, and the conductivity change $\sigma-\sigma_0$ is assumed to occur  with some distance to the boundary. In that case, the difference of the resulting potentials $v:=u_{\sigma}-u_{\sigma_0}$ solves
\begin{eqnarray}
\labeq{shunt1_diff} \nabla \cdot (\sigma_0 \nabla v) = - \nabla \cdot (\sigma-\sigma_0) \nabla u_\sigma \quad  \mbox{ in } \Omega,\\
\labeq{shunt2_diff} \int_{\mathcal E_l} \sigma \partial_\nu v \dx[s] =0 \quad \mbox{ for } l=1,\ldots,m,\\
\labeq{shunt3_diff} \sigma \partial_\nu v =0 \quad  \mbox{ on } \partial \Omega\setminus \bigcup_{l=1}^m \mathcal E_l,\\
\labeq{shunt4_diff} v|_{\mathcal E_l}=\mathrm{const.} \quad \forall l=1,\ldots,m,
\end{eqnarray}
If $\sigma_0$ is constant in a neighborhood of the boundary, then, 
 in the limit of point electrodes (cf.\ \cite{hanke2011justification}), $v$ solves
\[
\Delta v=0  \quad \mbox{ in a neighborhood of $\partial \Omega$}
\]
with homogeneous Neumann boundary data $\partial_\nu v|_{\partial \Omega}=0$.
Hence, by standard elliptic regularity results, we can expect the voltage difference to be a very smooth function on $\partial \Omega$. This justifies to interpolate the missing or erroneous values on current-driven electrodes by the neighboring values on current-free electrodes.

\subsection{Linear interpolation}\label{subsect:linear_interpolant}

We will first describe a simple linear interpolation approach that we will use as a benchmark for the more sophisticated approach described below.
In a situation where the $j$-th electrode is positioned between and spatially close to the $(j-1)$-th and the $(j+1)$-th electrode,
we may expect that the voltage difference $V_{j,j}$ should be close to the arithmetic mean of the values for the two neighboring electrode pairs, $V_{j,j-1}$ and $V_{j,j+1}$. Together with the fact that the measurement matrix should be symmetric and columns should sum up to zero, 
we thus obtain the following linear interpolation strategy.

\begin{definition}\label{def_interpolant}
Given a symmetric EIT difference measurement matrix $V\in \R^{m\times m}$ (with possibly erroneous entries $V_{jk}$ for $|j-k|\leq 1$), we 
define the 
\emph{linearly interpolated measurement matrix} $V^{\mathrm{lin}}\in \R^{m\times m}$
by solving the following equation system
\begin{eqnarray*}
V^{\mathrm{lin}}_{jk}=V_{jk} & \mbox{ for $|j-k|\geq 1$ (modulo $m$),}\\
\sum_{j=1}^m V_{jk}^{\mathrm{lin}}=0 \quad & \mbox{ for } k=1,\ldots,m,\\
V^{\mathrm{lin}}_{j,j-1}=V^{\mathrm{lin}}_{j-1,j}  \quad & \mbox{ for } j=1,\ldots,m,\\
V^{\mathrm{lin}}_{j,j}= \frac{1}{2}( V^{\mathrm{lin}}_{j-1,j} + V^{\mathrm{lin}}_{j,j+1} )
 \quad & \mbox{ for } j=1,\ldots,m.\\
\end{eqnarray*}
The linear equations are uniquely solvable for odd $m\in \N$. In the case of 
even $m\in \N$, the equations can be solved using the pseudo inverse.
\end{definition}
\subsection{Interpolation using geometry-specific smoothness}\label{subsect:interpolant_geometry}

Our numerical examples in the next section show that the linearly interpolated measurements converge against the true 
measurements. However, this convergence is rather slow. Also, if the electrodes are not positioned in a single circle around an imaging subject,
the above linear interpolation method seems questionable and a more geometry-specific method should be used.

Our new interpolation approach is based on a more careful study of the arguments in subsection \ref{subsect:motivation_interpolation}. 
The voltage difference $v=u_\sigma-u_{\sigma_0}$ solves the elliptic equation \req{shunt1_diff} with a source term on the support of the conductivity change $\sigma-\sigma_0$. The smoothness of $v|_{\partial \Omega}$ (and hence of the voltage measurements)
will depend on how far the support of $\sigma-\sigma_0$ is located away from the boundary $\partial \Omega$. We will make use of this geometry-specific smoothness
by assuming that we know an upper bound of the support, and demand that the interpolated measurements are consistent with a source term on this upper bound.

To make this idea precise, assume that we know a closed set $B$ (with non-empty interior) so that
\[
\supp(\sigma-\sigma_0)\subseteq B \subseteq \Omega.
\]
The $j$-th column in the matrix $V\in \R^{m\times m}$ contains the
difference measurements
\[
V_{jk}=v^{(k)}|_{\mathcal E_j}-v^{(k)}|_{\mathcal E_{j+1}},
\]
where $v^{(k)}$ solves
\begin{equation}\labeq{source_F}
\nabla \cdot (\sigma_0 \nabla v^{(k)}) = \nabla \cdot F^{(k)}, 
\end{equation}
with $F^{(k)}:=-(\sigma-\sigma_0) \nabla u_\sigma^{(k)}\in L^2(B)^n$, and the boundary conditions \req{shunt2_diff}--\req{shunt4_diff}.

Our new interpolation technique aims to fill in the missing data in such a way that \req{source_F} is fulfilled with the smallest possible source term $F^{(j)}\in L^2(B)$. Together with the symmetry and zero column sum condition, 
we thus obtain the following geometry-specific interpolation strategy.

\begin{definition}\label{def:main}
Let $V\in \R^{m\times m}$ be a symmetric EIT difference measurement matrix (with possibly erroneous entries $V_{jk}$ for $|j-k|\leq 1$).
We say that $F^{(1)},\ldots,F^{(m)}\in L^2(B)^n$ are \emph{minimal measurement-consistent source terms} if they 
minimize 
\[
\sum_{j=1}^m \norm{F^{(j)}}^2_{L^2(B)^n} \to \mbox{min!} 
\]
under the constraint that the measurements constructed from these sources fulfill the interpolation, symmetry and zero column sum condition, i.e.
with respective solutions $v^{(k)}$ of $\nabla \cdot (\sigma_0 \nabla v^{(k)}) = \nabla \cdot F^{(k)}$, \req{shunt2_diff}--\req{shunt4_diff},
\[
V^{\mathrm{geom}}_{jk}:=v^{(k)}|_{\mathcal E_j}-v^{(k)}|_{\mathcal E_{j+1}},
\]
must fulfill
\begin{eqnarray}
\labeq{geom_interp1} V^{\mathrm{geom}}_{jk}=V_{jk} & \mbox{ for $|j-k|\geq 1$ (modulo $m$),}\\
\labeq{geom_interp2} \sum_{j=1}^m V_{jk}^{\mathrm{geom}}=0 \quad & \mbox{ for } k=1,\ldots,m,\\
\labeq{geom_interp3} V^{\mathrm{geom}}_{j,j-1}=V^{\mathrm{geom}}_{j-1,j}  \quad & \mbox{ for } j=1,\ldots,m.
\end{eqnarray}

The matrix $V^{\mathrm{geom}}$ constructed from measurement-consistent source terms
is called the \emph{geometrically interpolated measurement matrix}.
\end{definition}

\begin{remark}
Definition \ref{def:main} can be interpreted as a minimum norm interpolation in problem-specific abstract smoothness classes. Let us introduce 
\[
\tilde K_B:\ L^2(B)^n\to L^2(\partial \Omega),\quad F\mapsto  f:=v|_{\partial \Omega},
\]
where $v\in H^1(\Omega)$ solves the source problem
\begin{equation}\labeq{rem_smoothness_source}
\nabla \cdot (\sigma_0 \nabla v) = \nabla \cdot F, 
\end{equation}
with boundary conditions \req{shunt2_diff}--\req{shunt4_diff}. Applying $\tilde K_B$ entrywise to a $m$-tuple of functions in $L^2(B)^n$ we obtain the operator
\[
K_B:\ (L^2(B)^n)^m\to L^2(\partial \Omega)^m,\quad \left( F^{(k)}\right)_{k=1}^m\mapsto \left( \tilde K_B (F^{(k)})\right)_{k=1}^m,
\]
and define the space
\[
H_B:=\mathcal{R}(K_B)\subseteq L^2(\partial \Omega)^m.
\]
with the range space norm
\[
\norm{\left( f^{(k)} \right)_{k=1}^m}_{H_B}:=\inf \{ \norm{\left( F^{(k)}\right)_{k=1}^m}_{(L^2(B)^n)^m}:\  
f^{(k)}=\tilde K_B (F^{(k)})\ \forall k \}.
\]

For two open sets $B_1$, $B_2$, with $\overline B_1\subseteq B_2$, a standard mollification argument shows
that every solution of \req{rem_smoothness_source} with source term $F\in L^2(B_1)^n$ also solves
\req{rem_smoothness_source} with a source term in $H^1(B_2)^n$. Since $H^1(B_2)^n$ is compactly embedded in
$L^2(B_2)^n$, it follows that the range space $H_{B_1}$ is compactly embedded in $H_{B_2}$.
Hence, analogously to the use of source conditions in regularization theory, 
the spaces $H_B$ can be interpreted as abstract smoothness classes which formalize the intuitive fact
that smaller sources generate smoother boundary potentials.

In that formalism, the geometrically interpolated measurement matrix
in definition \ref{def:main} is obtained from evaluating the minimum norm interpolant in $H_B$ under
the additional constraint of a symmetry and zero column sum condition.
\end{remark}

We will now show that the geometrically interpolated matrix is unique and that it can be obtained in an easy and computationally cheap
way. To that end, we assume that $\sigma_0$ fulfills a unique continuation property (UCP), which means that every
solution of $\nabla \cdot (\sigma_0 \nabla u_0) = 0$ that is constant on an open subset of $\Omega$ must be constant on all of $\Omega$.
Also, we assume that $B$ is chosen consistent with the pixel partition, i.e.,
\[
B=\bigcup_{P_i\subseteq B} P_i.
\]
As in subsection \ref{subsec:matrix_methods}, $S_i\in \R^{m\times m}$ denotes the sensitivity matrix for the $i$-th pixel, i.e.\ the $i$-th column of the sensitivity matrix written as a $m\times m$-matrix, cf.~\req{sensitivity_element_wise}. Summing up all elements in $B$ we define
\[
S_B=\sum_{i:\ P_i\subseteq B} S_i\in \R^{m\times m}.
\]
$S_B^+\in \R^{m\times m}$ denotes the (Moore-Penrose-)pseudoinverse of $S_B$. $e_j\in \R^m$ is the $j$-th unit vector.


\begin{theorem}\label{thm:main}
There exists a unique geometrically interpolated measurement matrix $V^{\mathrm{geom}}\in \R^{m\times m}$. 
It is the unique minimizer of
\[
-\sum_{j=1}^m e_j^T (V^{\mathrm{geom}})^T S_B^+ V^{\mathrm{geom}} e_j \to \mathrm{min!}
\]
under the constraints \req{geom_interp1}--\req{geom_interp3}.
\end{theorem}
Theorem \ref{thm:main} is proven in the next subsection.

\begin{remark}\label{remark:main}
The minimization problem in theorem~\ref{thm:main} can be solved analytically. 
Let $V\in \R^{m\times m}$ be a symmetric EIT difference measurement matrix  (with possibly erroneous entries $V_{jk}$ for $|j-k|\leq 1$).

$V^{\mathrm{geom}}\in \R^{m\times m}$ fulfills the constraints \req{geom_interp1}--\req{geom_interp3}
if and only if $V^{\mathrm{geom}}_{jk}=V_{jk}$ for $|j-k|>1$, and there exists a vector $w\in \R^m$ with
\begin{equation}\labeq{V_geom_and_w}
V^{\mathrm{geom}}_{j-1,j}=w_j,\quad V^{\mathrm{geom}}_{j+1,j}=w_{j+1},\quad V^{\mathrm{geom}}_{j,j}=-\sum_{l:\ |l-j|>1} V_{jl}-w_j-w_{j+1}.
\end{equation}
Hence, we can write 
\[
V^{\mathrm{geom}} e_j=A^{(j)} w + b^{(j)}
\]
where $A^{(j)}\in \R^{m\times m}$ and $b^{(j)}\in \R^m$ are given by
\begin{eqnarray*}
A^{(j)}_{kl}=\delta_{k,j-1}\delta_{l,j}+\delta_{k,j+1}\delta_{l,j+1}-\delta_{k,j}\delta_{l,j}-\delta_{k,j}\delta_{l,j+1}\\
%
 b^{(j)}_k=\left\{ \begin{array}{l l} V_{jk} & \mbox{ for } |k-j|>1,\\
 0 & \mbox{ for } |k-j|=1,\\
 -\sum_{l:\ |l-j|>1} V_{jl}  & \mbox{ for } k=j.
 \end{array}\right.
\end{eqnarray*}

Thus, with $A:=-\sum_{j=1}^m (A^{(j)})^T S_B^+ A^{(j)}$ and $b:=-\sum_{j=1}^m (A^{(j)})^T S_B^+ b^{(j)}$, the minimization problem is equivalent to minimizing
\[
\frac{1}{2}w^T A w + b^T w \to \mbox{min!}
\]
for which the unique solution is given by $w=-A^{-1}b$. The values of $V^{\mathrm{geom}}$
are then determined by \req{geom_interp1} and \req{V_geom_and_w}.
\end{remark}

\begin{remark}
Given $r$ pixels and $m$ electrodes, the computational effort to compute $S_B$ is at most
$rm^2$, and all other computations in remark~\ref{remark:main} are of the order $O(m^3)$.
In a setting where the upper bound $B$ and the sensitivity matrix $S$ stay constant over
several time-steps, a factorization of the matrix $A$ can be precomputed. Thus the interpolation
can be carried out in $O(m^2)$ steps, which is neglectable to the computational cost of reconstructing
the conductivity. Hence, our new interpolation method is well suited for real-time imaging.
\end{remark}

\begin{remark}
Note that our interpolation method uses the sensitivity matrix which also appears in linearized reconstruction
methods. However, the interpolation method does not rely on a linearized approximation and interpolates the values of the true voltage difference (that depends non-linearly on the conductivity change). 
\end{remark}

\subsection{Proof of theorem~\ref{thm:main}}
The rest of this section is devoted to proving theorem~\ref{thm:main}. We first introduce the so-called
virtual measurement operator (cf., e.g., \cite{harrach2013recent,harrach2013monotonicity})
\[
L:\ L^2(B)^n\to \R^m,\quad F\mapsto \left( v|_{\mathcal E_k}-v|_{\mathcal E_{k+1}} \right)_{k=1}^m\in \R^m,
\]
where $v|_{\mathcal E_k}-v|_{\mathcal E_{k+1}}$, $k=1,\ldots,m$, is the measurement on the 
$k$-th electrode pair of the solution of the source problem
\[
\nabla \cdot (\sigma_0 \nabla v) = \nabla \cdot F, 
\]
with boundary conditions \req{shunt2_diff}--\req{shunt4_diff}. The next lemma summarizes some properties of $L$ and its connection to the sensitivity matrix.
\begin{lemma}\label{lemma:L}
\begin{enumerate}[(a)]
\item The adjoint of $L$ is given by
\[
L^*:\ \R^m\to L^2(B)^n, \quad L^* w=\sum_{j=1}^m w_j \nabla u_{\sigma_0}^{(j)}|_{B} \ \mbox{ for } w=(w_j)_{j=1}^m\in \R^m,
\]
where $u_{\sigma_0}^{(j)}$ solves \req{shunt1}--\req{shunt4} with reference conductivity $\sigma_0$ and electric current driven through the $j$-th and $(j+1)$-th electrode.
\item If $B$ is consistent with the pixel partition, i.e., $B=\bigcup_{i:\ P_i\subseteq B} P_i$, then $S_B=-L L^*$.
\item If $\sigma_0$ fulfills a UCP, then $\range(L)=\{ v:\  \sum_{j=1}^m v_j=0 \}\subset \R^m$.
%
\end{enumerate}
\end{lemma}
\begin{proof}
\begin{enumerate}[(a)]
\item For the $j$-th unit vector $e_j\in \R^m$ and $F\in L^2(B)^n$ we have that
\begin{eqnarray*}
\int_B F\cdot (L^* e_j) \dx = (LF)\cdot e_j= v|_{\mathcal E_j}-v|_{\mathcal E_{j+1}}=\int_{\partial \Omega} \sigma_0 \partial_\nu u_{\sigma_0}^{(j)}|_{\partial \Omega}\,  v|_{\partial \Omega} \dx[s]\\
= \int_B F \cdot \nabla u_{\sigma_0}^{(j)} \dx.
\end{eqnarray*}
\item For the $j$-th and $k$-th unit vector $e_j,e_k\in \R^m$, we have that
\begin{eqnarray*}
\int_B (L^* e_j)\cdot (L^* e_k) \dx = \int_B \nabla u_{\sigma_0}^{(j)}\cdot \nabla u_{\sigma_0}^{(k)}\dx\\
= \sum_{i:\ P_i\subseteq B} \int_{P_i} \nabla u_{\sigma_0}^{(j)}\cdot \nabla u_{\sigma_0}^{(k)}\dx = 
- \sum_{i:\ P_i\subseteq B} e_j^T S_i e_k^T =-e_j^T S_B e_k.
\end{eqnarray*}
\item If $\sigma_0$ fulfills a UCP, then
\[
0=L^* w=\sum_{j=1}^m w_j \nabla u_{\sigma_0}^{(j)}|_{B}
\]
is equivalent to $\sum_{j=1}^m w_j u_{\sigma_0}^{(j)}=\mathrm{const.}$ on all of $\Omega$. 

Due to the definition of $u_{\sigma_0}^{(j)}$, we have that for each $l=1,\ldots,m$,
\begin{eqnarray*}
\int_{\mathcal E_l} \sigma_0 \partial_\nu \sum_{j=1}^m w_j u_{\sigma_0}^{(j)} \dx[s]
= \sum_{j=1}^m w_j \int_{\mathcal E_l} \sigma_0 \partial_\nu u_{\sigma_0}^{(j)} \dx[s]
=\sum_{j=1}^m w_j (\delta_{j,l}-\delta_{j+1,l})\\
= w_l-w_{l+1}
\end{eqnarray*}
Hence, $\sum_{j=1}^m w_j u_{\sigma_0}^{(j)}=\mathrm{const.}$ is equivalent to $w_1=\ldots=w_m=\mathrm{const.}$, so that
the assertion follows from $\kernel(L^*)^\perp=\range(L)$.
%
\end{enumerate}
\end{proof}


\textit{Proof of theorem~\ref{thm:main}.}
Let $V\in \R^{n\times n}$ be symmetric difference EIT measurements (with possibly erroneous entries $V_{jk}$ for $|j-k|\leq 1$). 
Due to lemma~\ref{lemma:L}(c), the minimization problem in definition \ref{def:main} is equivalent to
finding $V^{\mathrm{geom}}\in \R^{m\times m}$ that minimizes
\[
\sum_{j=1}^m \norm{L^+ V^{\mathrm{geom}} e_j}_{L^2(B)^n}^2\to \mbox{min!}
\]
under the constraints \req{geom_interp1}--\req{geom_interp3}. 

From lemma~\ref{lemma:L}(b) it follows that
$(L^+)^* L^+=(LL^*)^+=-S_B^+$, so that
\begin{eqnarray*}
\norm{L^+ V^{\mathrm{geom}} e_j}^2_{L^2(B)^n}= -e_j^T (V^{\mathrm{geom}})^T S_B^+ V^{\mathrm{geom}} e_j.
\end{eqnarray*}
Finally, for all $v\in \range(L)$ we have that
\[
-v^T S_B^+ v = \norm{L^+ v}^2_{L^2(B)^n}\geq \frac{1}{\norm{L}^2} \norm{L L^+ v}_2^2
= \frac{1}{\norm{L}^2} \norm{v}^2_2,
\]
so that 
\[
-\sum_{j=1}^m e_j^T (V^{\mathrm{geom}})^T S_B^+ V^{\mathrm{geom}} e_j \geq \frac{1}{\norm{L}^2} \sum_{j=1}^m \norm{V^{\mathrm{geom}} e_j}^2
= \frac{1}{\norm{L}^2} \norm{V^{\mathrm{geom}}}_{\mathrm{Fro}}^2.
\]
This shows that $-\sum_{j=1}^m e_j^T (V^{\mathrm{geom}})^T S_B^+ V^{\mathrm{geom}} e_j$ is a strictly coercive quadratic 
functional on the affine linear space of matrices fulfilling the constraints \req{geom_interp1}--\req{geom_interp3}, which is a non-empty subset of $\range(L)$.
Hence, the minimization problem is uniquely solvable.\hfill $\Box$

\section{Numerical results}\label{sect:numerics}

In this section, we use the open source framework EIDORS developed by Adler and Lionheart \cite{adler2006uses}, available on \texttt{eidors.org}, to numerically evaluate how good the interpolated measurements agree with the true ones and to test the performance of the matrix-based monotonicity method on interpolated data.

In our first example, the imaging domain $\Omega$ is the two-dimensional unit circle, the reference conductivity is $\sigma_0=1$, and the true conductivity is $\sigma=1+\chi_{D_1}+\chi_{D_2}+\chi_{D_3}$ with a larger half-ellipsoidal inclusion $D_1$ and two smaller circular inclusions $D_2$ and $D_3$,
see figure~\ref{fig:2D_setting}. We use EIDORS to simulate difference EIT measurements $V:=U(\sigma)-U(\sigma_0)\in \R^{m\times m}$ for an adjacent-adjacent driving pattern with $m$ electrodes (including voltages on current-driven electrodes) as described in section \ref{Subsec:setting}. The sensitivity matrix is also obtained from EIDORS but with another FEM grid to avoid an inverse crime. Figure \ref{fig:2D_setting} shows $\sigma_0$ and $\sigma$ for a setting with $m=32$ electrodes, and the FEM grids used for calculating the measurements $U(\sigma_0)$, $U(\sigma)$ and the sensitivity matrix $S$. The FEM grid for the latter is also used as the pixel partition for the reconstruction. Note that the choice of the reconstruction grid is ad-hoc in our examples in this section. For a fixed reconstruction grid, the monotonicity method will
show the support of the conductivity difference up to the grid resolution in the limit of noiseless data and infinitely many electrodes. But it not yet clear how to choose an optimal reconstruction grid for a given amount of noise and number of electrodes (see however \cite{harrach2015resolution} for a method to evaluate whether certain reconstruction guarantees hold true on a given grid).

\begin{figure}
\begin{center}
\begin{tabular}{c c c}
\mbox{\includegraphics[height=4.5cm,trim=140 215 100 250,clip]{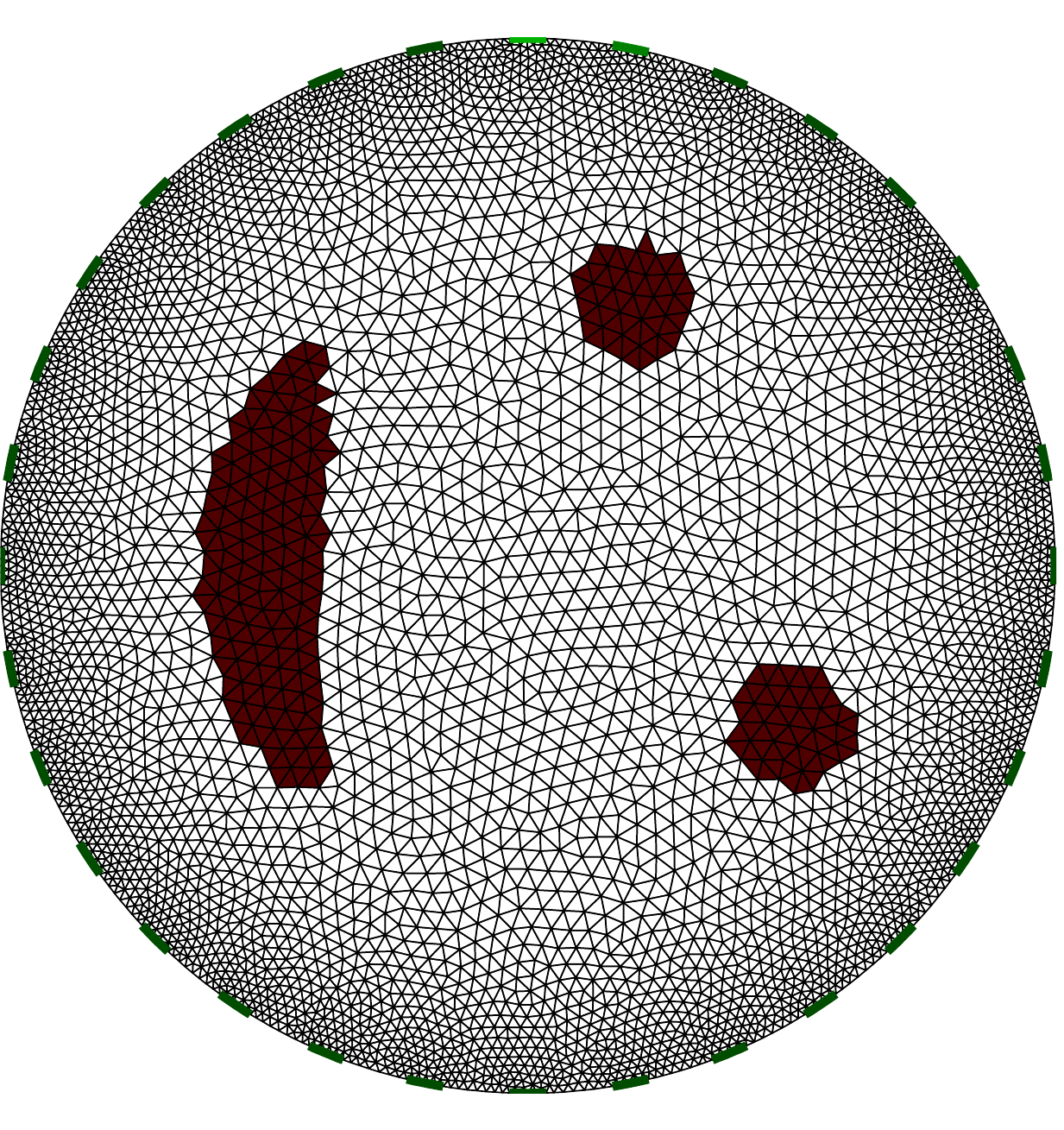}} &
\mbox{\includegraphics[height=4.5cm,trim=140 215 100 250,clip]{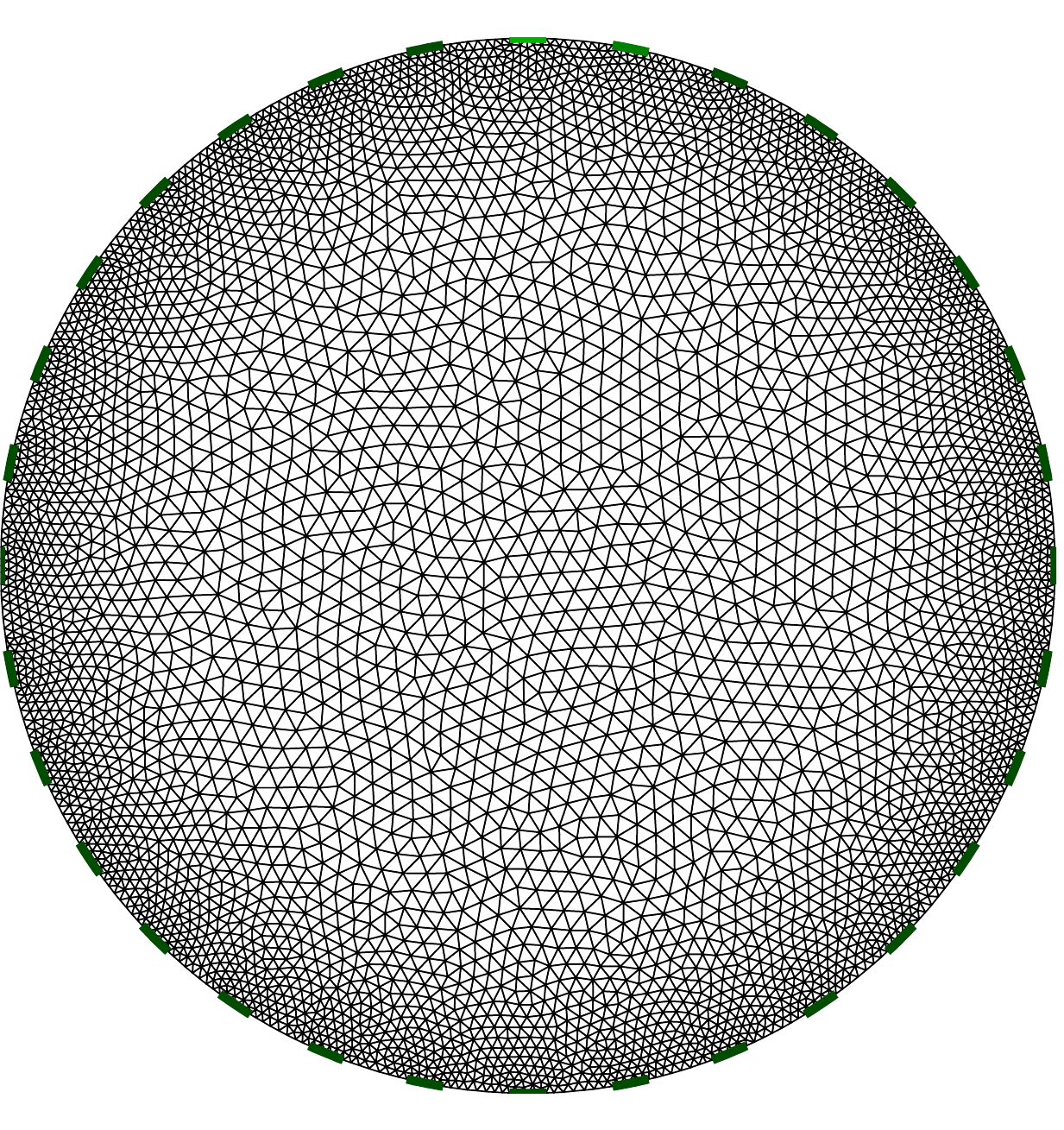}} &
\mbox{\includegraphics[height=4.5cm,trim=140 215 100 250,clip]{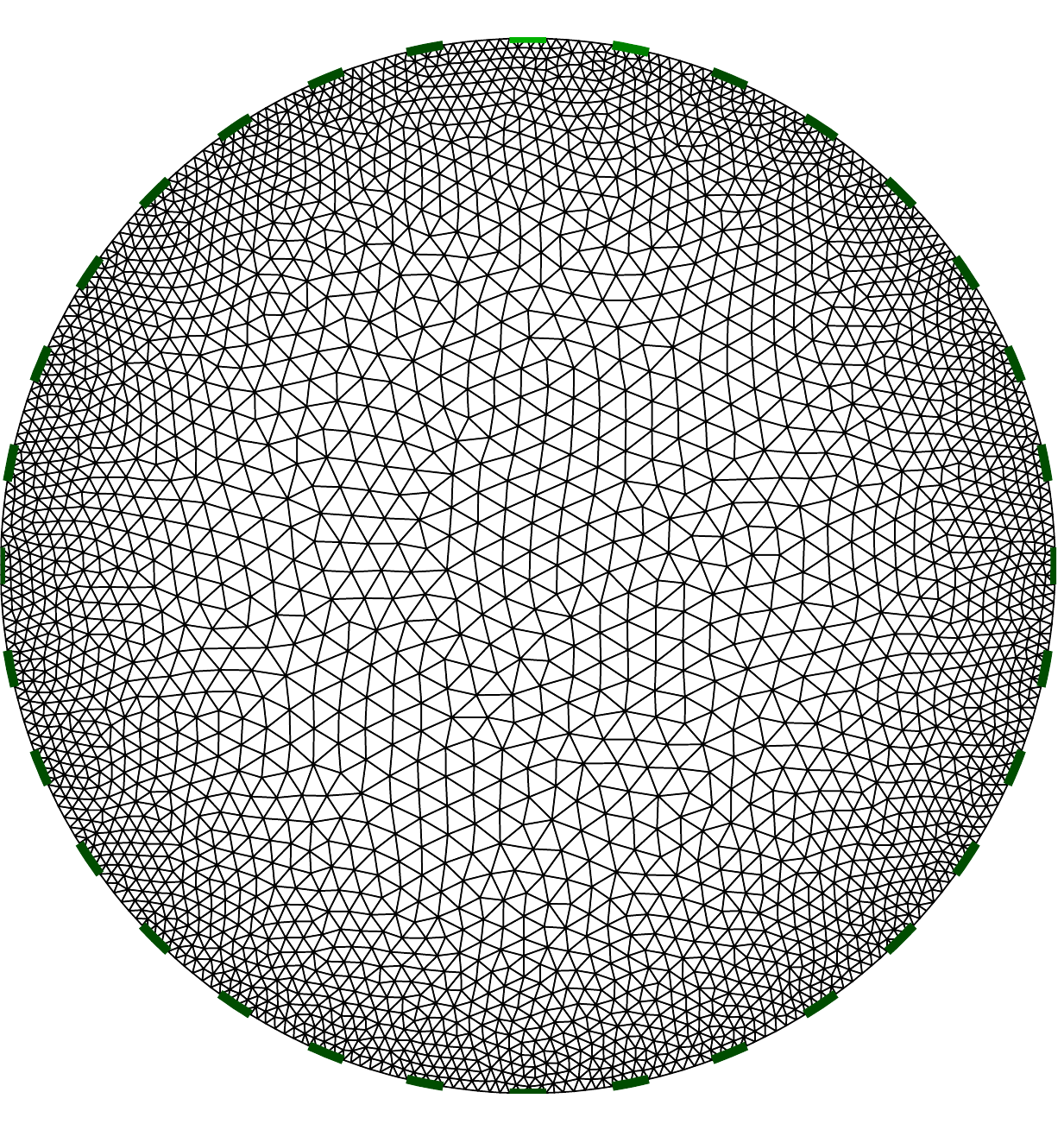}}
\end{tabular}
\end{center}
\caption{True conductivity $\sigma$ (left image) and reference conductivity $\sigma_0$ (middle image)
and the FEM grid used to calculate $U(\sigma)$ and $U(\sigma_0)$. 
The right image shows the FEM grid used to calculate the sensitivity matrix $S$.}
\label{fig:2D_setting}
\end{figure}

From the simulated EIT measurements $V\in \R^{m\times m}$, the entries corresponding to voltages on current-driven electrodes, $V_{jk}$ with $|j-k|<1$ (modulo $m$), are removed, and replaced using the simple linear interpolation method in subsection \ref{subsect:linear_interpolant}, and by our new geometrical interpolation method in subsection \ref{subsect:interpolant_geometry}. For the latter, we use as upper bound $B$ the union
of all pixels which intersect an open ball (centered at the origin) with radius $r>0$.

For a first impression of the performance of our new interpolation method consider figure \ref{fig:2D_electrodes_error}.
The black circles show the first 50 voltages taken columnwise from the difference EIT measurements $V$ for $m=24$ electrodes.
The 24th entry corresponds to the voltage between the 24th and the 1st electrode with current applied between the 1st and 2nd electrode,
the 25th-27th entry are the voltages between the 1st and the 2nd, the 2nd and the 3rd, and the 3rd and the 4th
electrode, resp., while the current is driven between the 2nd and the 3rd electrode. Accordingly, these four entries are
voltages on current driven electrodes. The red dashed line shows the values obtained by linear interpolation and
the black solid line shows the values obtained by our new geometric interpolation method with $r=0.7$. 
Our new method recovers the voltage on active electrodes much more precisely than the linear interpolation method.

\begin{figure}
\begin{center}
\mbox{\includegraphics[height=5cm,trim=70 410 40 240,clip]{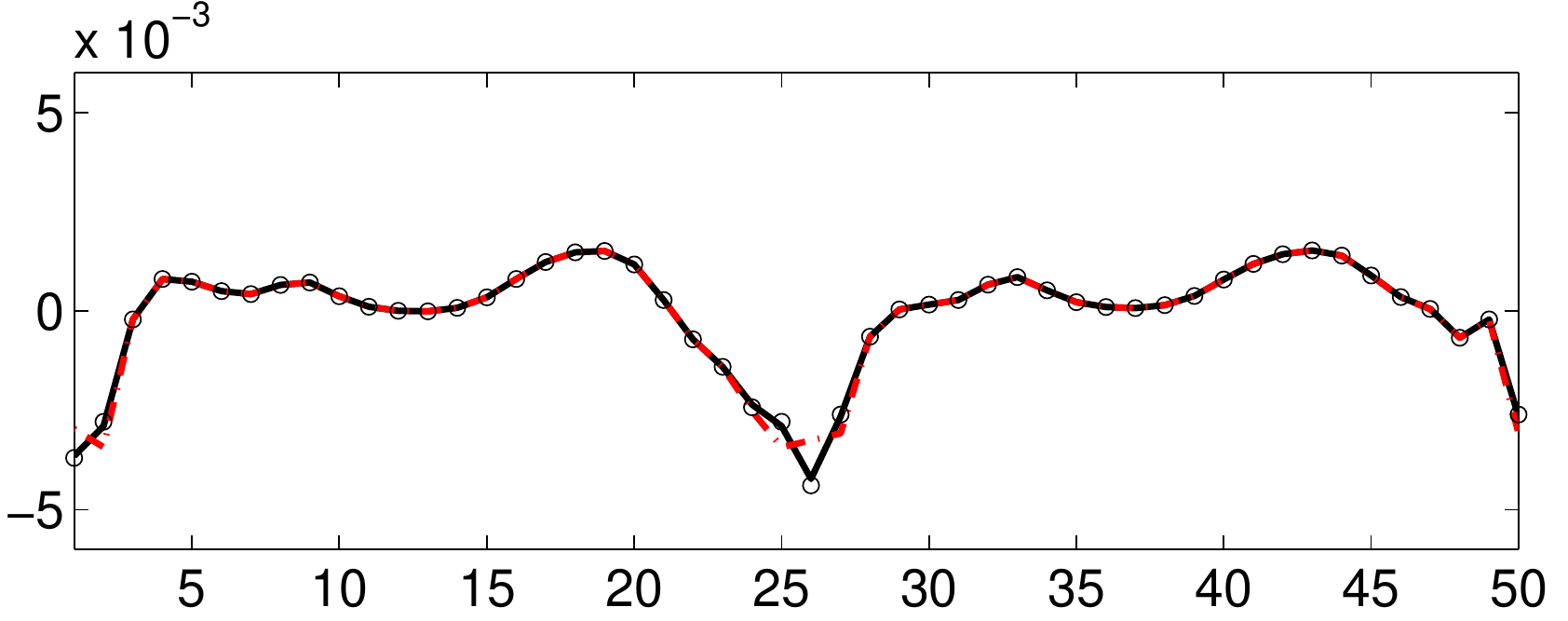}}
\end{center}
\caption{First 50 voltages measured in a setting with $m=24$ electrodes including voltages on 
current driven electrodes (black circles) and the interpolated values using simple linear interpolation
(red dashed line) and the new geometric interpolation method with $r=0.7$ (black line).}
\label{fig:2D_electrodes_error}
\end{figure}

Table \ref{table:2D_interpolation_error} lists the relative interpolation errors in the Frobenius norm
\[
\frac{\norm{V-V^\mathrm{lin}}_F}{\norm{V}_F}, \quad \mbox{ resp., } \quad \frac{\norm{V-V^\mathrm{geom}}_F}{\norm{V}_F}
\]
for different electrode numbers $m$ and radii $r$ of the upper bound $B$ .
The results for our new interpolation using geometry-specific smoothness are clearly superior to linear interpolation. The interpolation 
error falls below $1\%$ already for $32$ electrodes and an upper bound $B$ with $r=0.8$.

\begin{table}
\begin{center}
\begin{tabular}{l || r | r | r | r | r}
& $m=8$ & $m=16$ & $m=24$ & $m=32$ & $m=40$\\ \hline \hline
linear interpolation      & 75.44\% & 32.29\% & 15.65\% & 8.23\% & 5.22\% \\ \hline
geom.~interpol. ($r=0.9$) & 79.07\% & 31.03\% & 11.00\% & 3.77\% & 1.48\% \\  \hline
geom.~interpol. ($r=0.8$) & 69.33\% & 18.09\% &  3.98\% & 0.67\% & 0.21\% \\ \hline
geom.~interpol. ($r=0.7$) & 61.88\% & 12.17\% &  2.48\% & 0.41\% & 0.16\% 
\end{tabular}
\end{center}
\caption{Relative error (measured in the Frobenius norm) caused by replacing voltages on current driven
electrodes by interpolated values.}
\label{table:2D_interpolation_error}
\end{table}

We now study the feasibility of using interpolated measurements for the matrix-based monotonicity method
described in subsection \ref{subsec:matrix_methods}.
We simulate the measurement matrix $V\in \R^{m\times m}$ for $m=32$ electrodes (including the voltages on current-driven electrodes). Using a matrix $E \in \R^{m\times m}$ containing independently uniformly distributed random values in the interval $[-1,1]$, we simulate noisy measurements by setting
\[
V_\delta:=V+\delta \norm{V}_F \, \frac{E}{\norm{E}_F}
\]
for a given relative noise level $\delta>0$. As above, we then remove the voltages on current-driven electrodes
in $V_\delta$ and replace them by interpolated values using linear interpolation and our new geometry-specific interpolating method (again with $r=0.7$, $r=0.8$, and $r=0.9$) to obtain the matrix $V^\mathrm{lin}_\delta$, resp., $V^\mathrm{geom}_\delta$. Our aim is to compare how much the monotonicity-based methods described in subsection \ref{subsec:matrix_methods} are affected by using the interpolated
measurements $V^\mathrm{lin}_\delta$, resp., $V^\mathrm{geom}_\delta$ instead of $V_\delta$.

We implement the monotonicity method's indicator function \req{monotonicity_indicator} in the following form. 
As in \req{sensitivity_element_wise}, let $S_i\in \R^{m\times m}$ contain the entries of the $i$-th column of the sensitivity-matrix $S\in \R^{m^2\times r}$. Since we deal with noisy measurements, we calculate the following regularized version of \req{beta_montonicity}
\begin{eqnarray*}\labeq{beta_delta}
\beta_i^\delta:=\max \{\beta\geq 0:\   \beta S_i\geq -|V_\delta| -\delta \norm{V_\delta}_F I \},
\end{eqnarray*}
so that $\beta_i^\delta>0$ is guaranteed to hold for every pixel inside the inclusions.

To calculate $\beta_i^\delta$, note that $|V| +\delta \norm{V_\delta}_F I$ is positive definite
and thus possesses an invertible matrix square root $A\in \R^{m\times m}$ with $A^*A=|V_\delta| +\delta \norm{V_\delta}_F I$.
For each $\beta>0$, we have that 
\[
\beta S_i + |V_\delta| +\delta \norm{V_\delta}_F I = \beta S_i + A^* A = A^*(A^{-*} \beta S_i A^{-1} + I)A,
\]
so that the definiteness condition $\beta S_i\geq -|V| -\delta \norm{V_\delta}_F I$ is equivalent to
\begin{equation}\labeq{beta_calc_hilf}
\beta  A^{-*} S_i A^{-1} + I\geq 0.
\end{equation}
Since obviously $S_i\leq 0$, the condition \req{beta_calc_hilf} is fulfilled if and only if no eigenvalue
of $A^{-*} S_i A^{-1}$ is smaller (more negative) than $-\frac{1}{\beta}$. Hence,
\[
\beta_i^\delta=-\frac{1}{\lambda_1}
\]
where $\lambda_1$ is the smallest (most negative) eigenvalue of $A^{-*} S_i A^{-1}\in \R^{m\times m}$.

In the same way, we calculate the monotonicity indicator functions
for the interpolated measurements $V^\mathrm{lin}_\delta$ and $V^\mathrm{geom}_\delta$. 
Figure \ref{fig:2D_reconstructions_beta_32} shows (for $m=32$ electrodes)
the plot of $x\mapsto \sum_{j=1}^r \beta_i^\delta\chi_i(x)$ for measurements including the correct
values on current-driven electrodes (first column), 
for geometrically interpolated measurements (with $r=0.7$, $r=0.8$ and $r=0.9$ in the second, third and fourth column, respectively),
and for linearly interpolated measurements (fifth column). The first line in figure \ref{fig:2D_reconstructions_beta_32}
corresponds to the almost noiseless case ($\delta=0.001\%$), for the second line we added relative noise with $\delta=0.1\%$.
The reconstructions are plotted with the EIDORS standard setting that all pixels with values less than $25\%$ of the maximal value
are set to transparent color. Additionally we cropped the color scale so that all pixels with values more than $50\%$ of the maximal value
are plotted with the same color. The results indicate that measurements on current-driven electrodes can well be replaced by interpolated measurements.

\begin{figure}
\begin{center}
\begin{tabular}{c c c c c}
\mbox{\includegraphics[height=2.6cm,trim=140 215 100 250,clip]{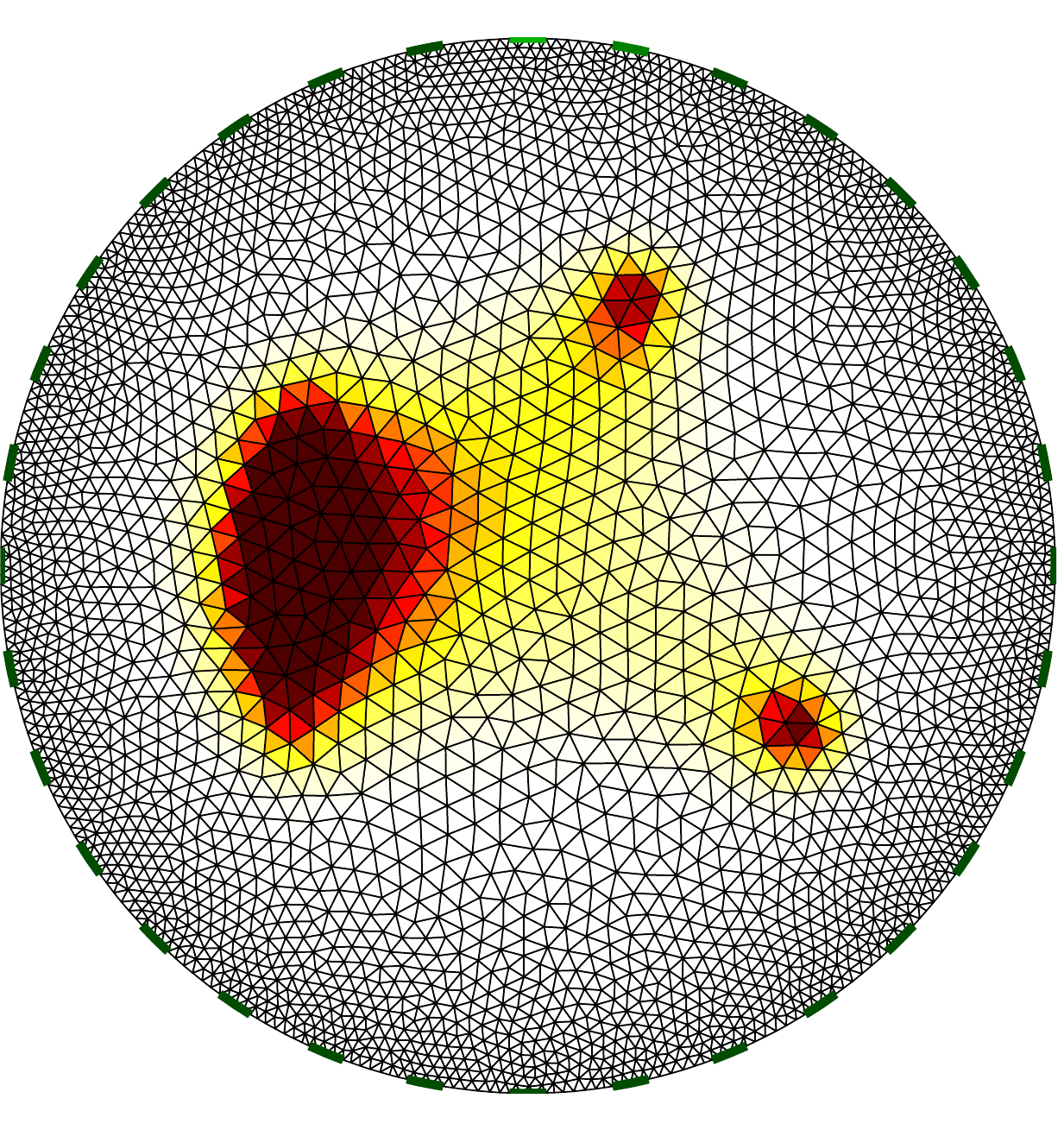}} &
\mbox{\includegraphics[height=2.6cm,trim=140 215 100 250,clip]{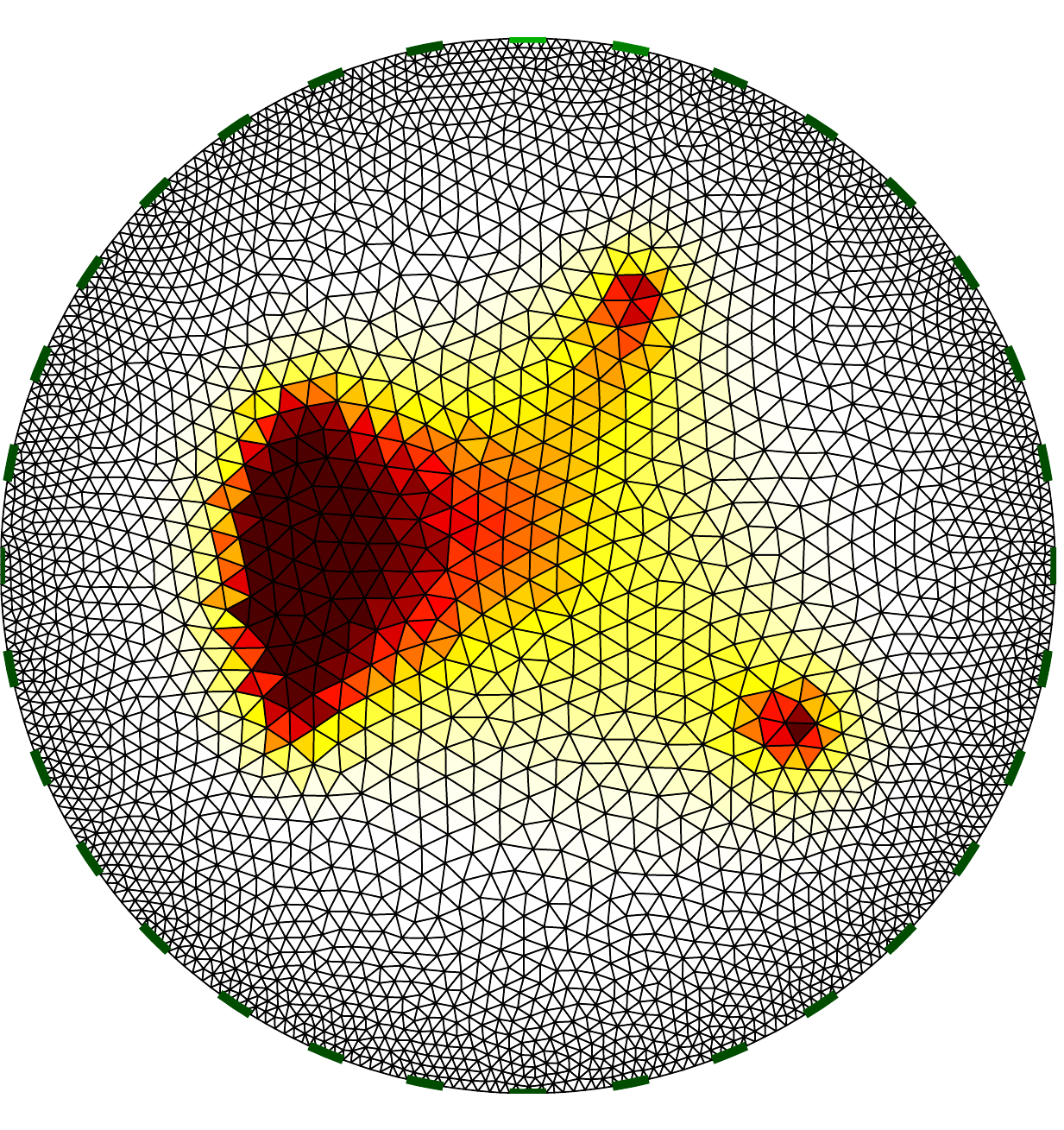}} &
\mbox{\includegraphics[height=2.6cm,trim=140 215 100 250,clip]{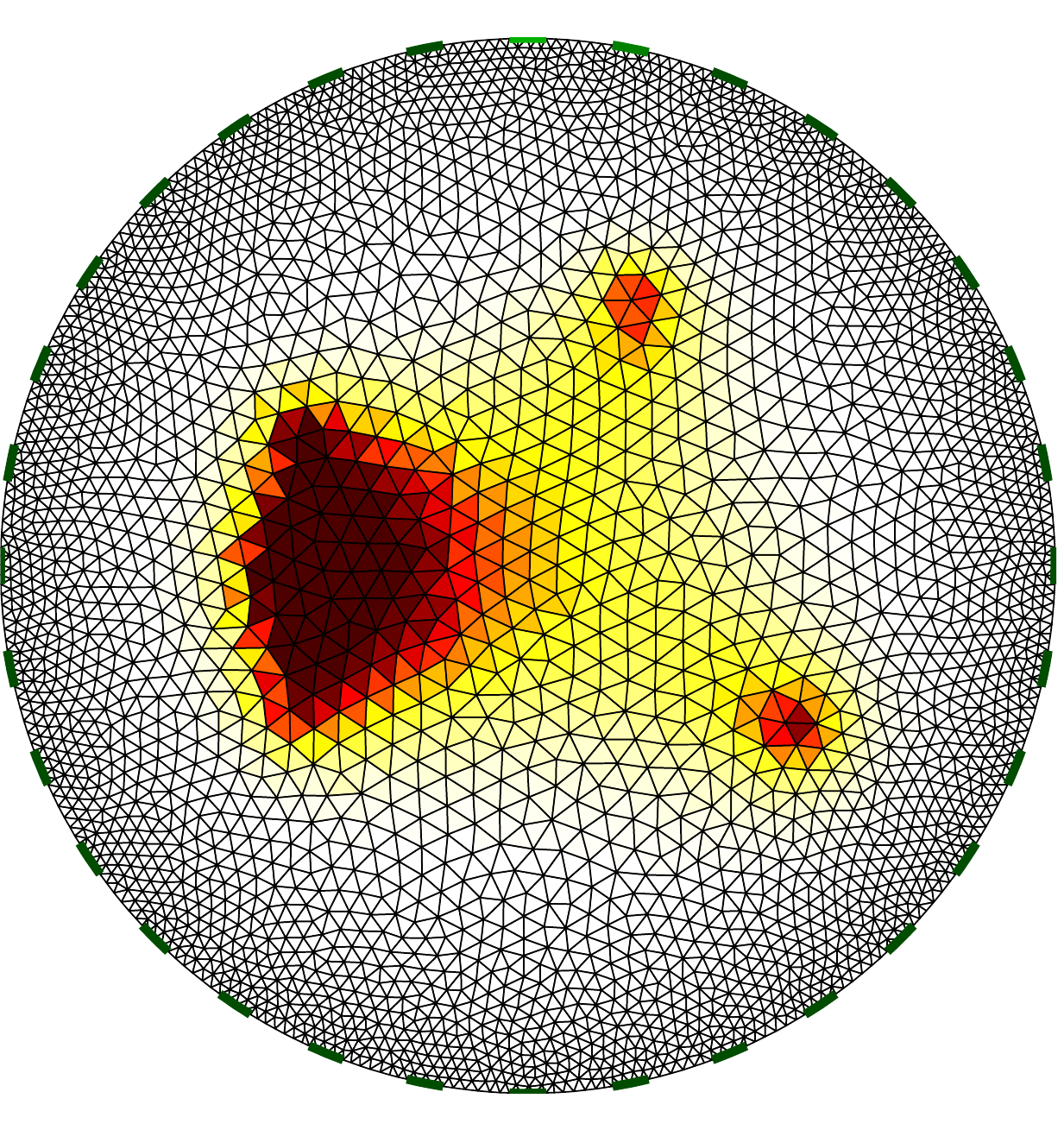}} &
\mbox{\includegraphics[height=2.6cm,trim=140 215 100 250,clip]{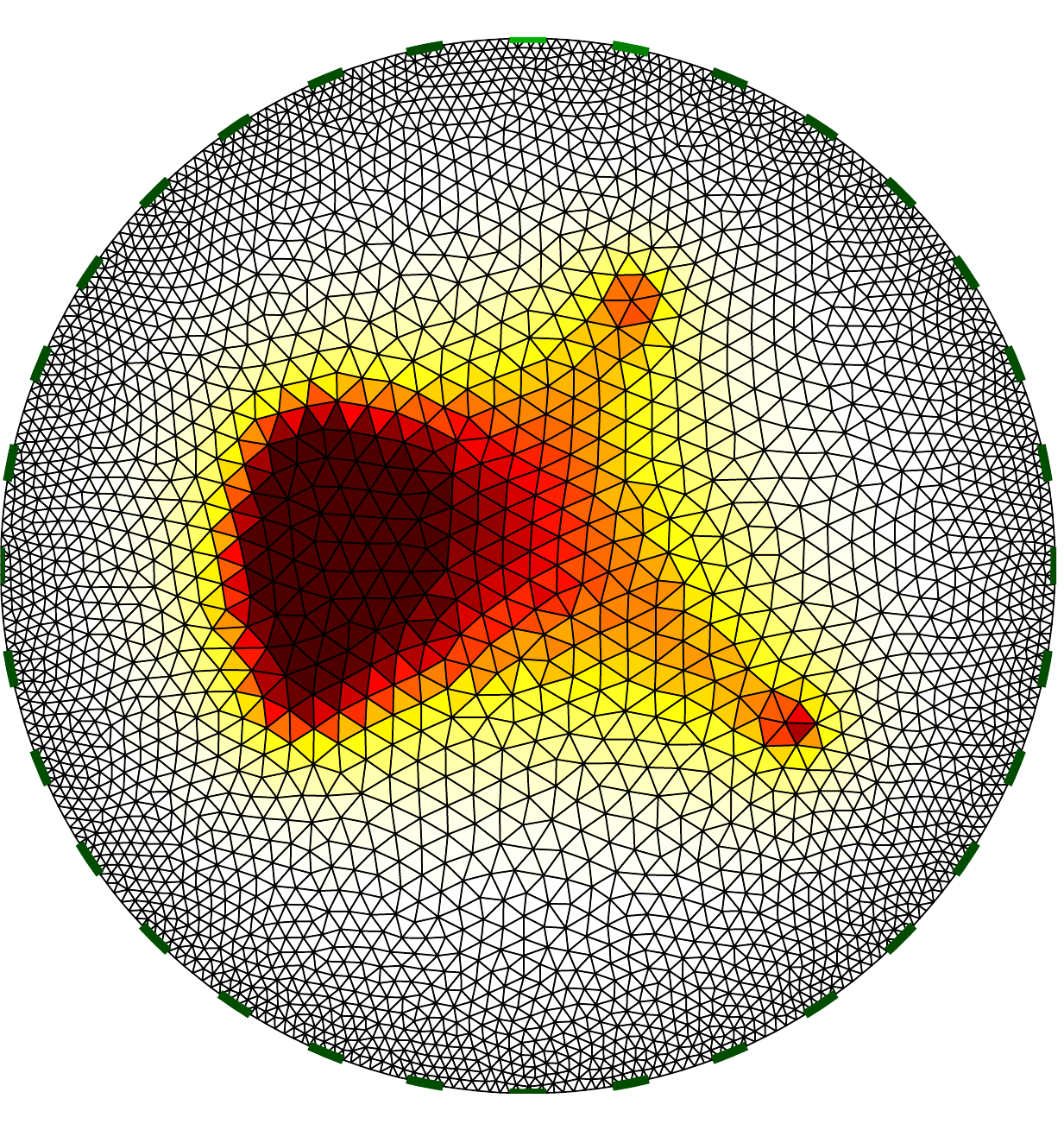}} &
\mbox{\includegraphics[height=2.6cm,trim=140 215 100 250,clip]{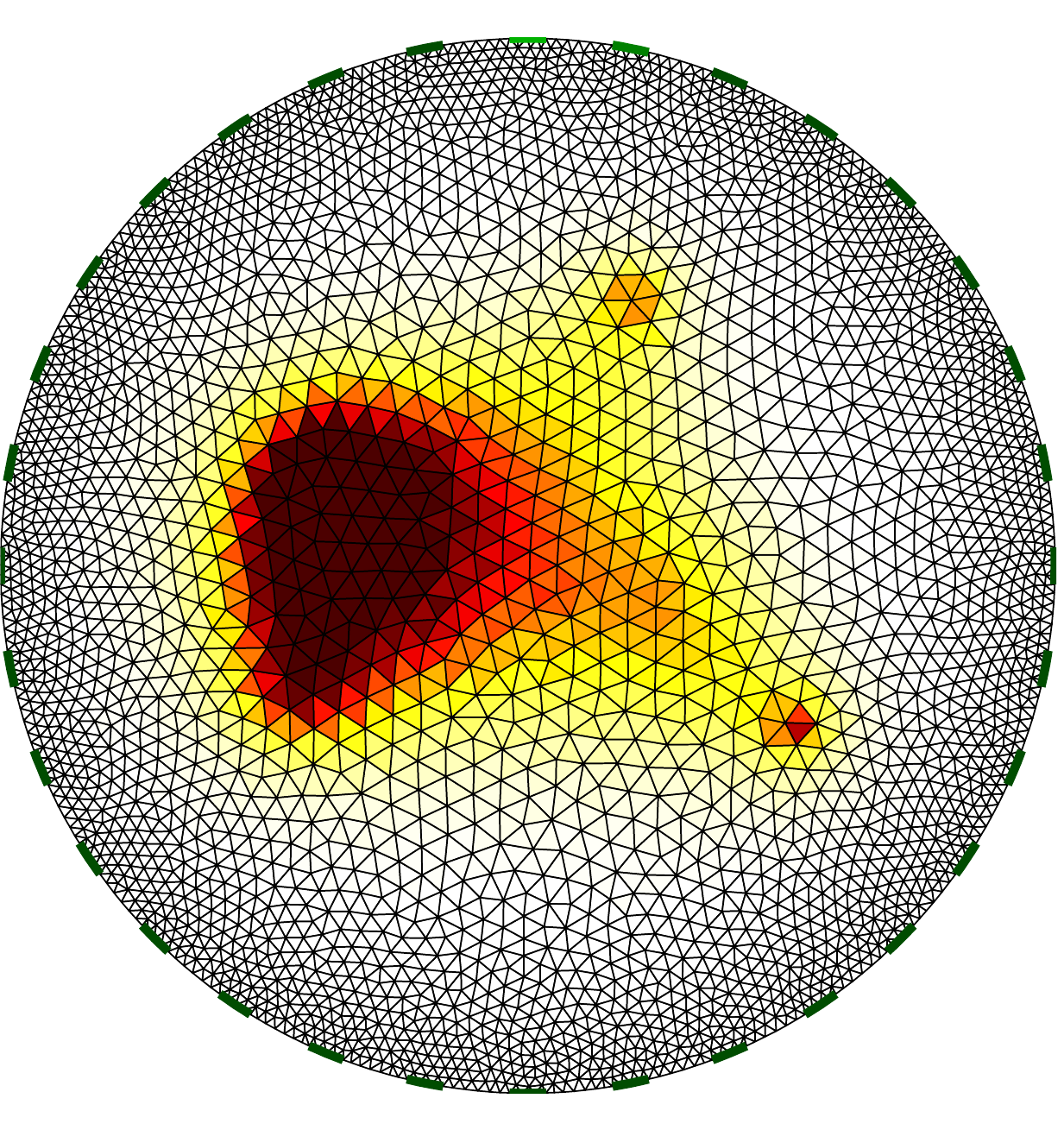}}\\
\mbox{\includegraphics[height=2.6cm,trim=140 215 100 250,clip]{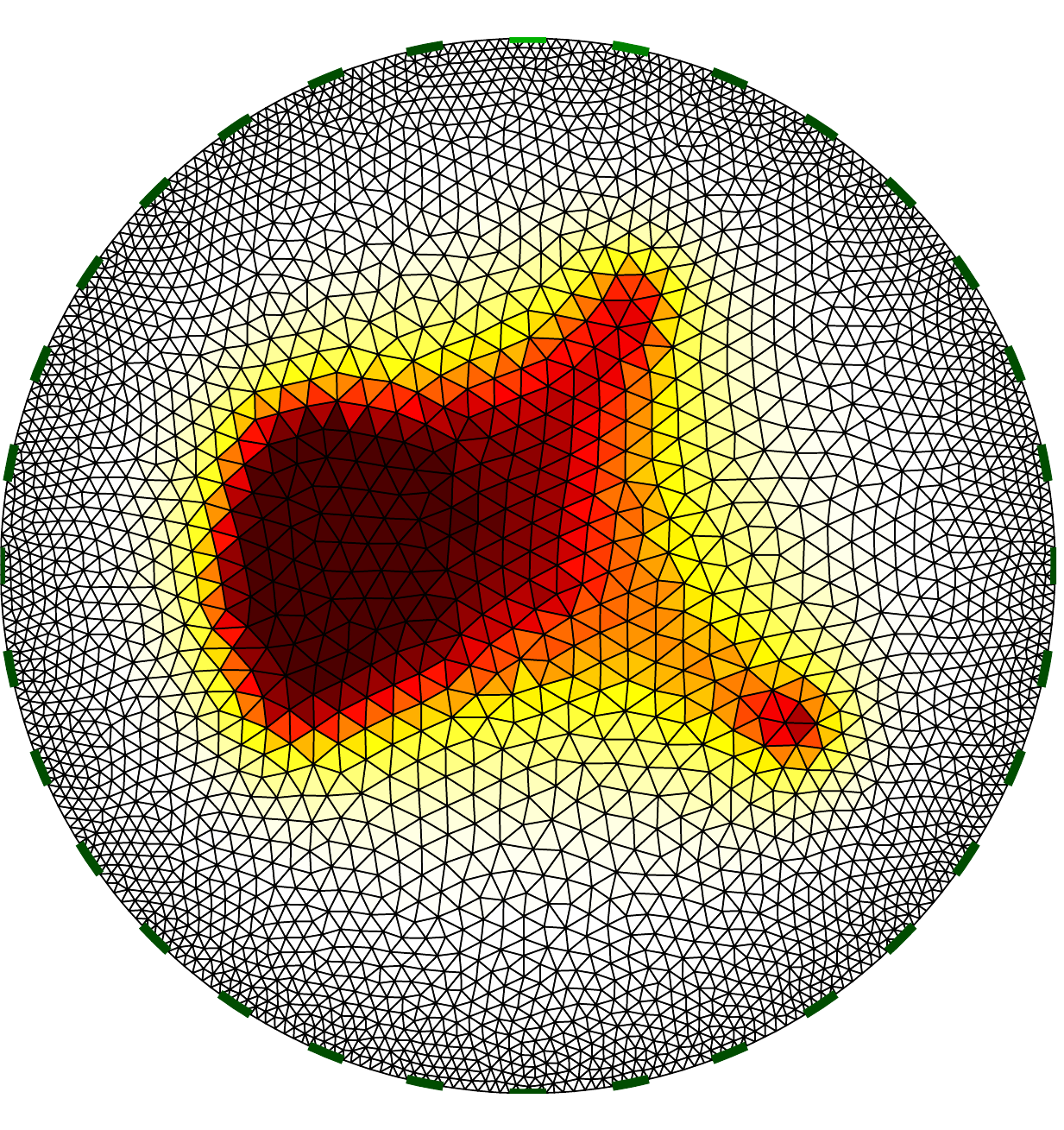}} &
\mbox{\includegraphics[height=2.6cm,trim=140 215 100 250,clip]{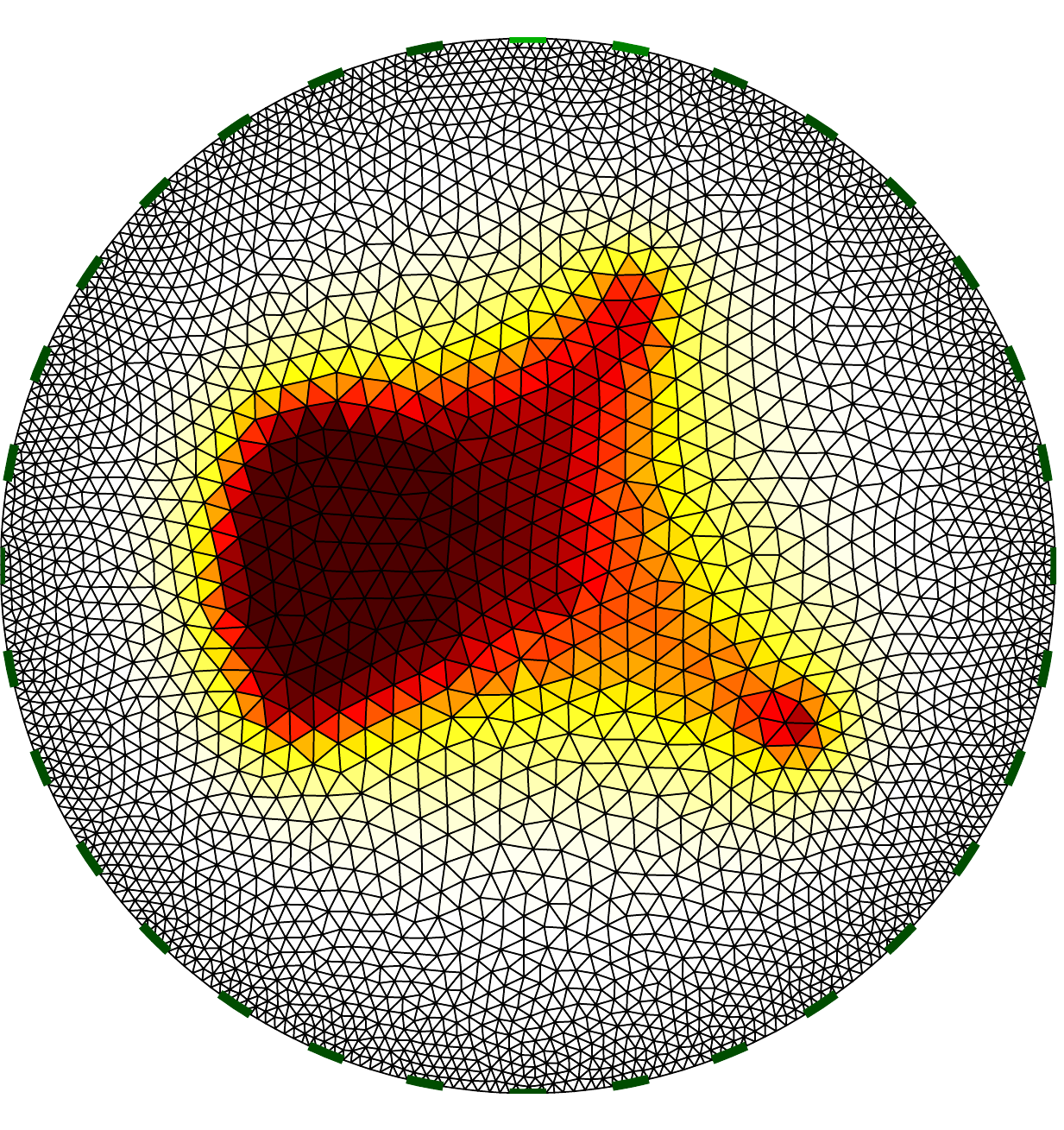}} &
\mbox{\includegraphics[height=2.6cm,trim=140 215 100 250,clip]{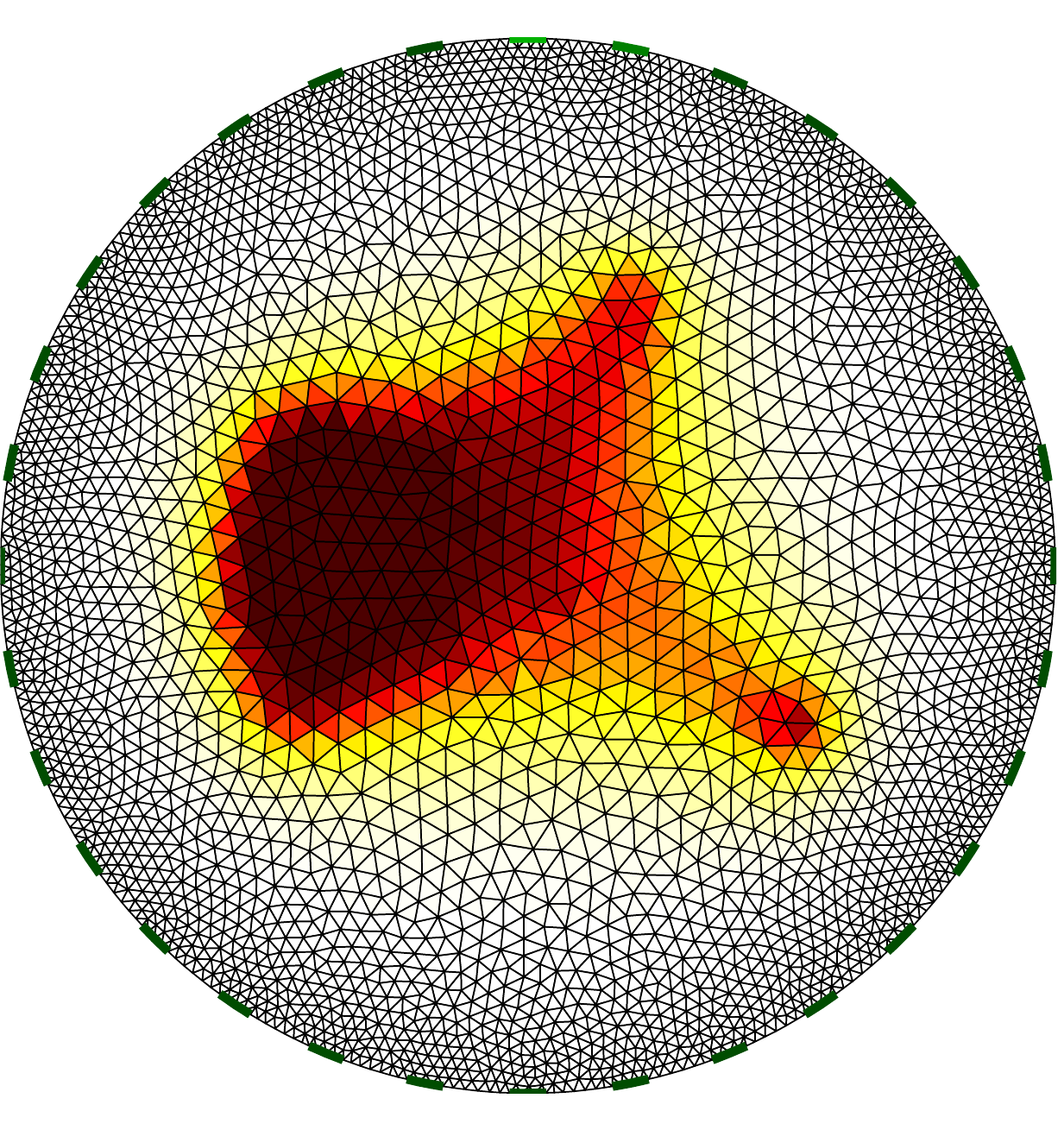}} &
\mbox{\includegraphics[height=2.6cm,trim=140 215 100 250,clip]{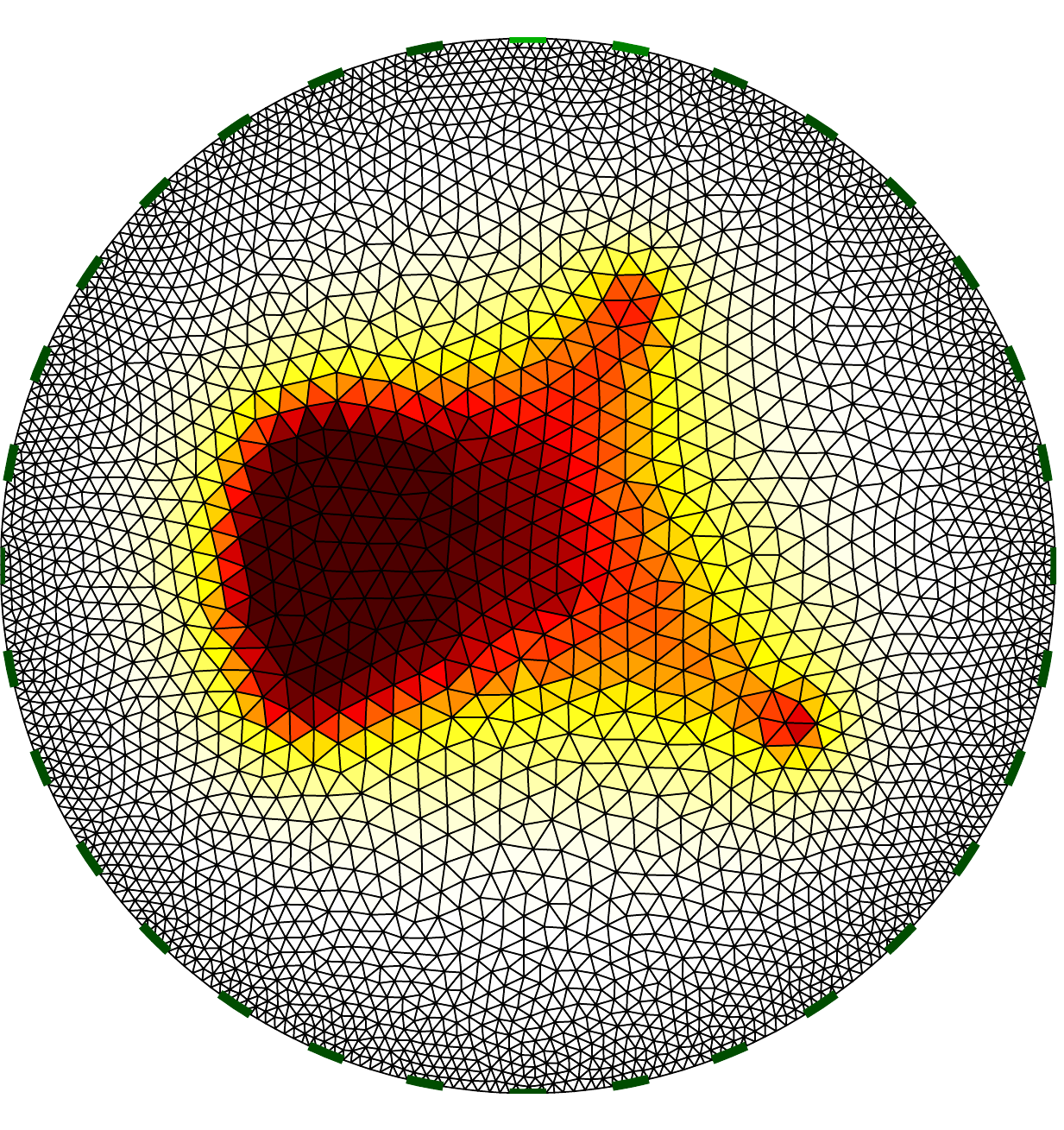}} &
\mbox{\includegraphics[height=2.6cm,trim=140 215 100 250,clip]{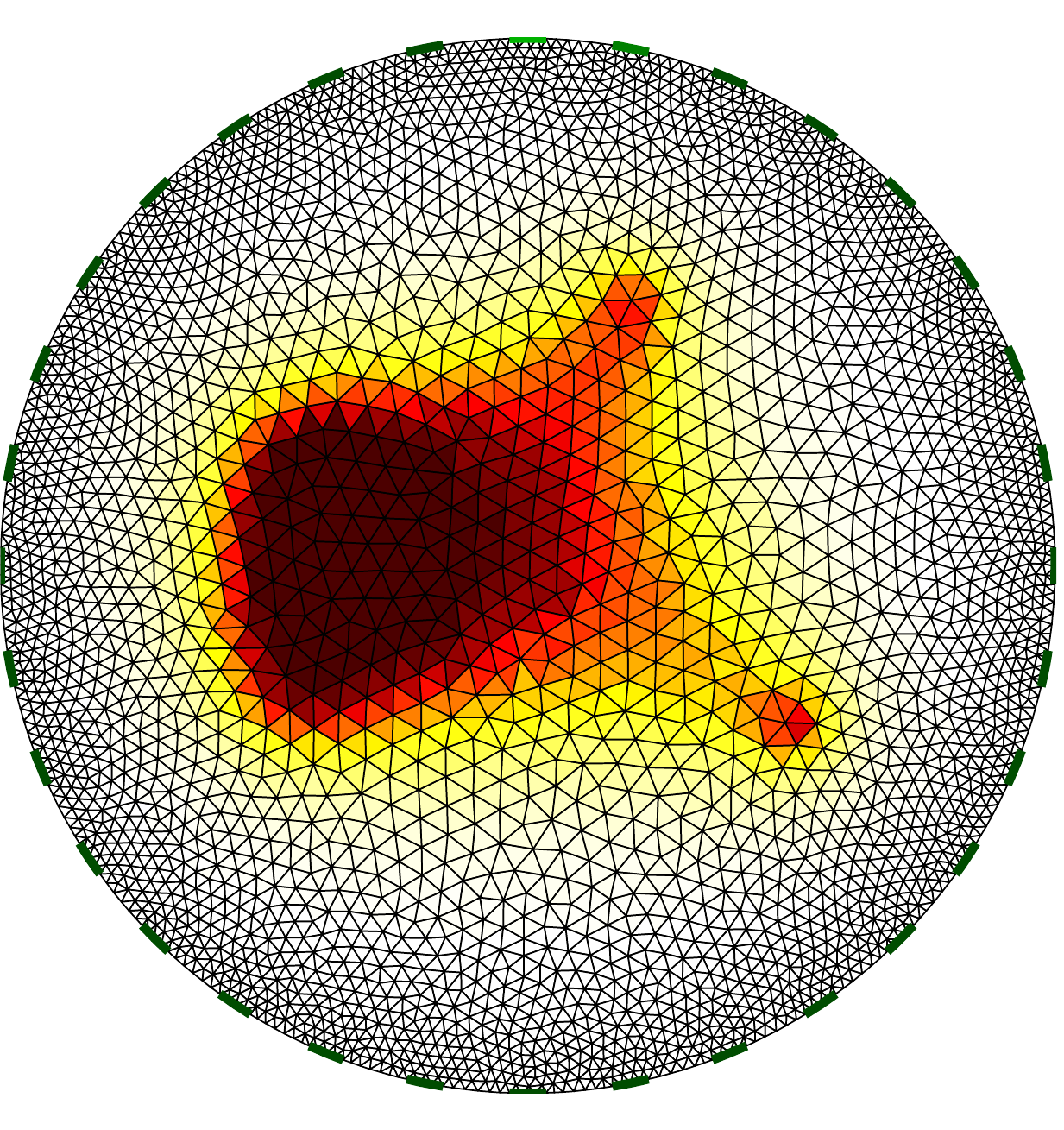}}
\end{tabular}
\end{center}
\caption{Monotonicity-based reconstructions using measurements on $32$ electrodes including voltages on current driven electrodes (1st column) and after replacing them with geometrically interpolated values (with $r=0.7$, $r=0.8$ and $r=0.9$ in the 2nd, 3rd and 4th column),
and with linearly interpolated values (5th column). First line contains almost noiseless data ($\delta=0.001\%$), second line contains relative noise with $\delta=0.1\%$. 
}
\label{fig:2D_reconstructions_beta_32}
\end{figure}

We also test our new interpolation method on a 3D example, where the imaging domain $\Omega$ is a cylinder to which electrodes are 
attached in $3$ planes of $k$ electrodes each. The reference conductivity is $\sigma_0=1$, and the true conductivity is $\sigma=1+\chi_{D_1}+\chi_{D_2}$ with a larger, vertically aligned, cylindrical inclusion $D_1$ and a smaller circular inclusion $D_2$,
see figure~\ref{fig:3D_setting}. We used EIDORS to simulate difference EIT measurements $V:=U(\sigma)-U(\sigma_0)\in \R^{m\times m}$ for an adjacent-adjacent driving pattern with $m=3k$ electrodes, including voltages on current-driven electrodes. Note that the driving patterns
also contain measurements between electrodes on different planes. 
Figure \ref{fig:3D_setting} shows $\sigma_0$ and $\sigma$ for a setting with $k=24$ electrodes on each plane, and the FEM grids used for calculating the measurements $U(\sigma_0)$, $U(\sigma)$ and for calculating the sensitivity matrix $S$. 

\begin{figure}
\begin{center}
\begin{tabular}{c@{\qquad} c@{\qquad} c}
\mbox{\includegraphics[height=4cm,trim=200 245 160 280,clip]{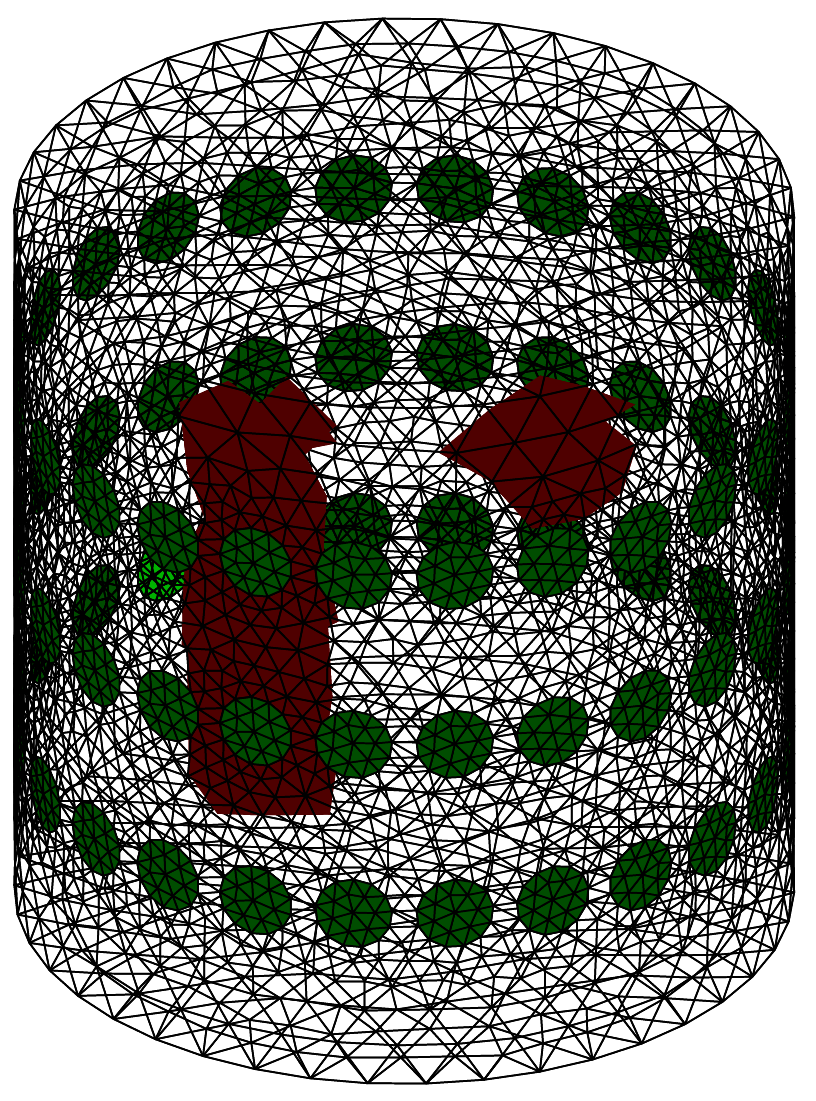}} &
\mbox{\includegraphics[height=4cm,trim=200 245 160 280,clip]{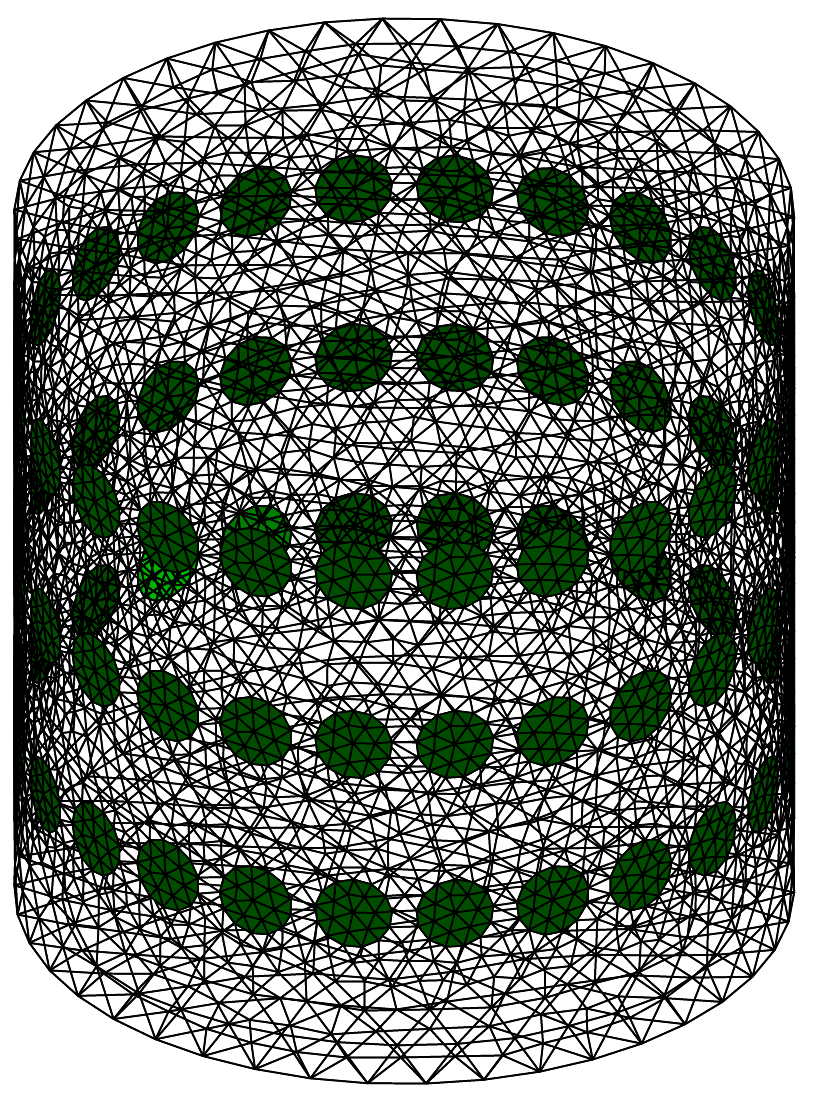}} &
\mbox{\includegraphics[height=4cm,trim=200 245 160 280,clip]{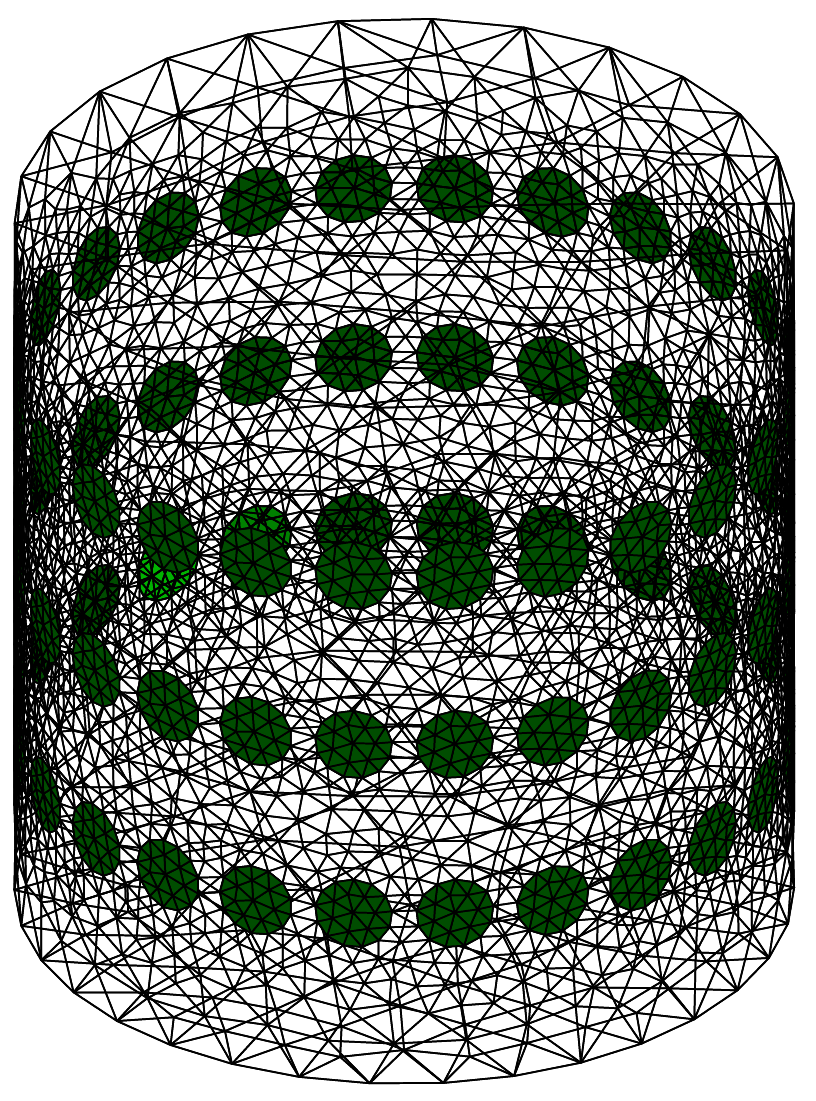}}
\end{tabular}
\end{center}
\caption{True conductivity $\sigma$ (left image) and reference conductivity $\sigma_0$ (middle image)
and the FEM grid used to calculate $U(\sigma)$ and $U(\sigma_0)$. 
The right image shows the FEM grid used to calculate the sensitivity matrix $S$.}
\label{fig:3D_setting}
\end{figure}

As in the two-dimensional case, we calculate the relative interpolation error that arises when measurements on current driven electrodes
are removed and replaced using our new geometrical interpolation method. For the latter, we use as upper bound $B$ the union of all pixels which intersect an open cylinder (centered at the origin) with radius $r=0.7$, $r=0.8$, and $r=0.9$.
Table \ref{table:3D_interpolation_error} lists the relative interpolation errors in the Frobenius norm for a
total number of $3\cdot 16$, $3\cdot 20$, $3\cdot 24$, and $3\cdot 28$ electrodes. The naive application of our simple linear interpolation method
fails on measurements between electrodes on different planes, and produced interpolation errors roughly around $40\%$ for all our settings.
Our new geometric interpolation method approximates the voltages on current driven electrodes well, the accuracy 
when using 3 planes of 28 electrodes is around $2.5\%$.

\begin{table}
\begin{center}
\begin{tabular}{l || r | r | r | r}
& $m=3\cdot 16$ & $m=3\cdot 20$ & $m=3\cdot 24$ & $m=3\cdot 28$\\ \hline \hline
geom.~interpol. ($r=0.9$) & 37.90\% & 23.93\% & 16.74\% & 7.88\%\\  \hline
geom.~interpol. ($r=0.8$) & 23.17\% & 11.82\% &  7.52\% & 3.17\%\\ \hline
geom.~interpol. ($r=0.7$) & 17.51\% &  8.04\% &  6.81\% & 2.56\%
\end{tabular}
\end{center}
\caption{Relative error (measured in the Frobenius norm) caused by replacing voltages on current driven
electrodes by interpolated values.}
\label{table:3D_interpolation_error}
\end{table}

Figure \ref{fig:3D_reconstructions_beta} shows the results of using the monotonicity method
for the setting with $k=24$ electrodes on each of 3 planes. We only calculated the indicator $\beta_i$
for pixels between the upper and lower electrode plane. From left to right, the figure shows the reconstruction
from measurements including active electrode data, and after replacing them with geometrically interpolated values using
$r=0.7$, $r=0.8$ and $r=0.9$, and after replacing them with linearly interpolated values. A relative noise of $\delta=0.1\%$ was added to the measurements.
The reconstructions are plotted with the EIDORS standard setting that all pixels with values less than $25\%$ of the maximal value are set to transparent color. The monotonicity method works well with geometrically interpolated data. The reconstructions with interpolated data 
are almost indistinguishable from that using active electrode data, which is even better than what one would expect from the interpolation error in table~\ref{table:3D_interpolation_error}. Moreover, even with naively linearly interpolated data, a very rough estimate of the convex hull of the inclusions is obtained.

\begin{figure}
\begin{center}
\begin{tabular}{c c c c c}
\fbox{\includegraphics[height=3.3cm,trim=200 245 160 280,clip]{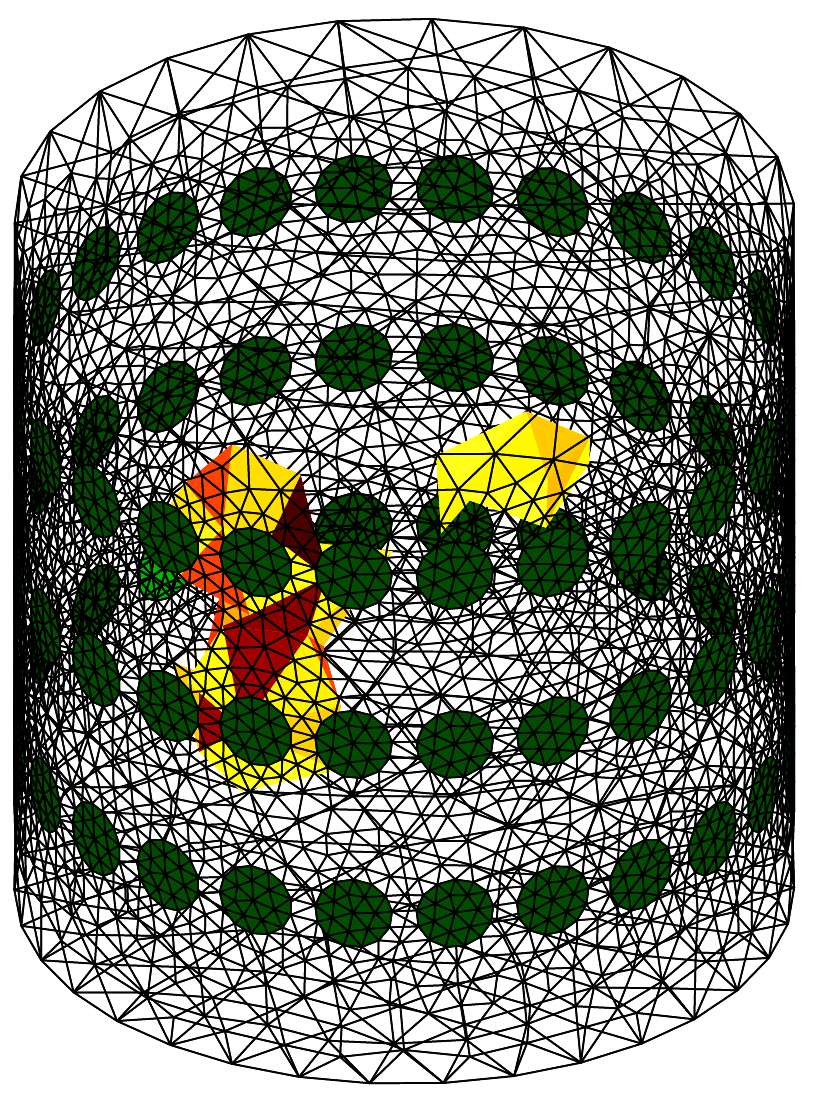}} &
\fbox{\includegraphics[height=3.3cm,trim=200 245 160 280,clip]{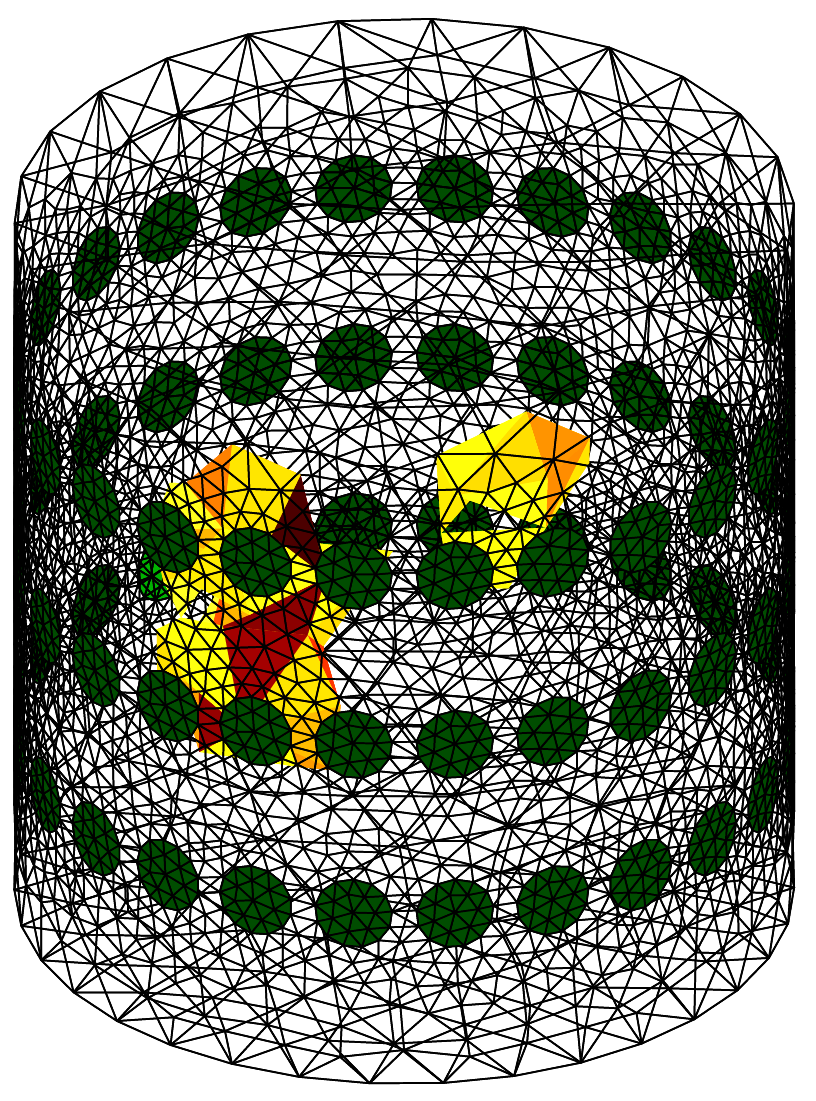}} &
\fbox{\includegraphics[height=3.3cm,trim=200 245 160 280,clip]{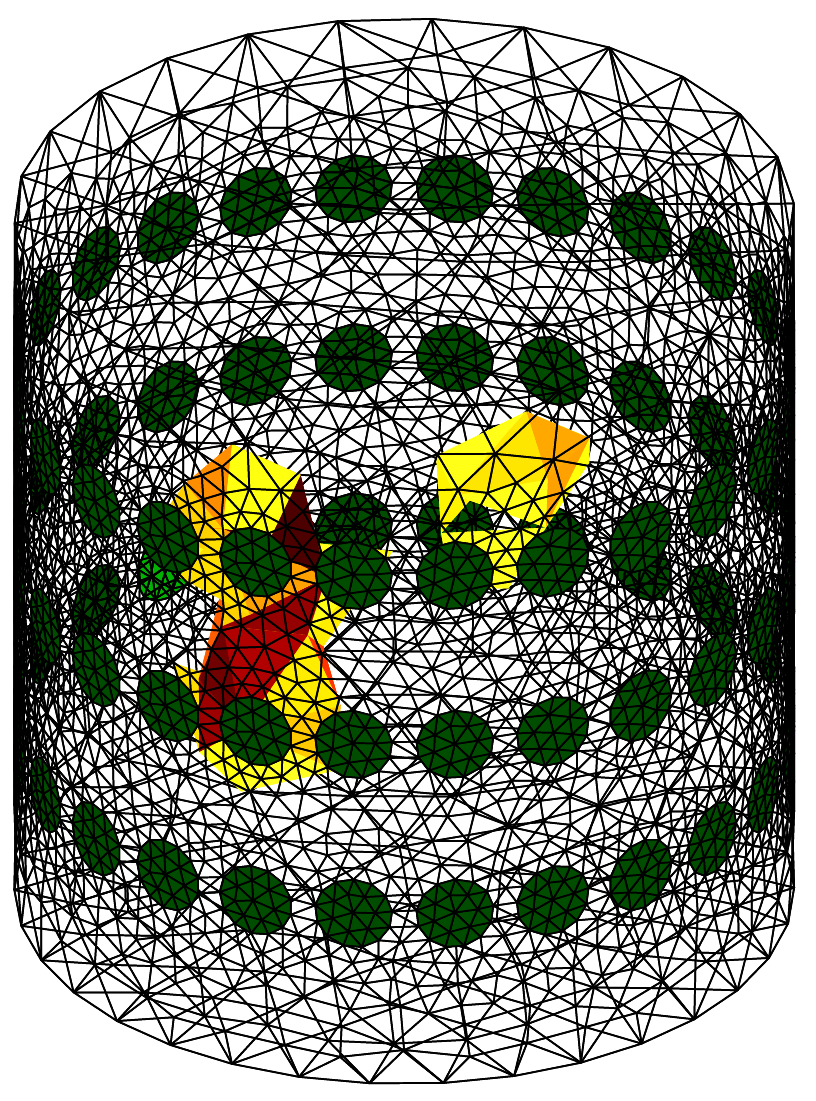}} &
\fbox{\includegraphics[height=3.3cm,trim=200 245 160 280,clip]{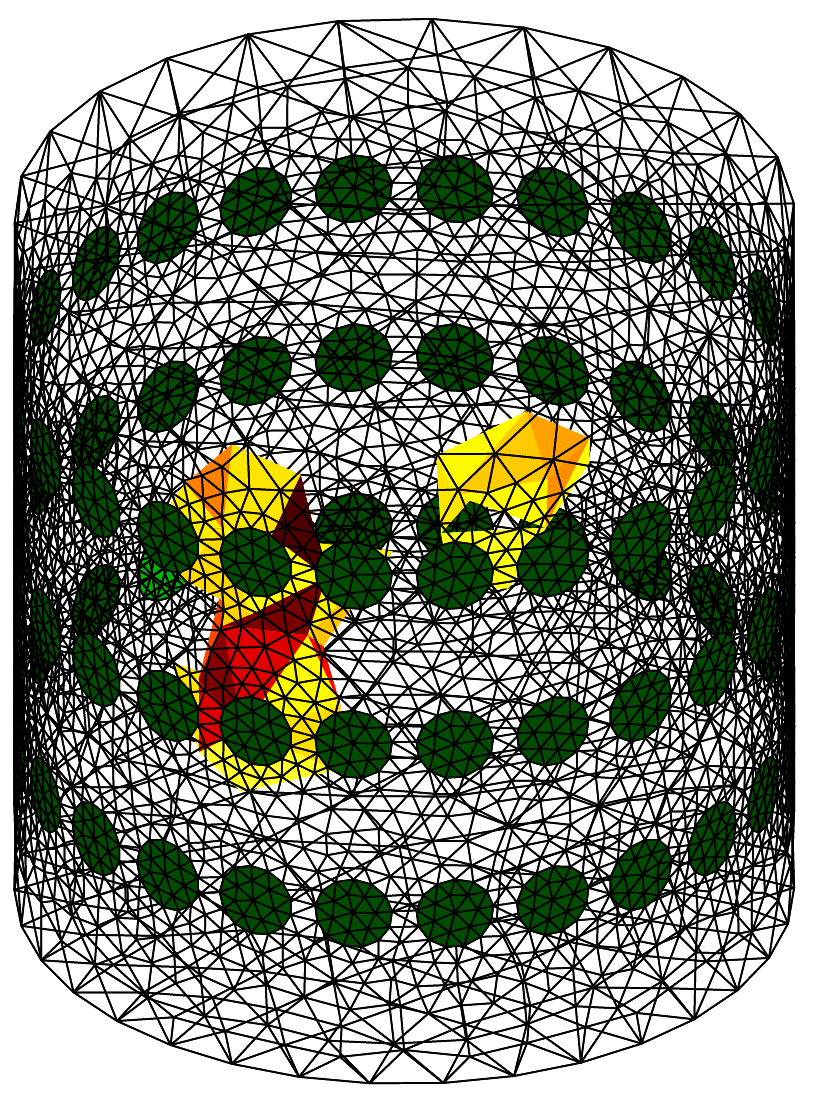}} &
\fbox{\includegraphics[height=3.3cm,trim=200 245 160 280,clip]{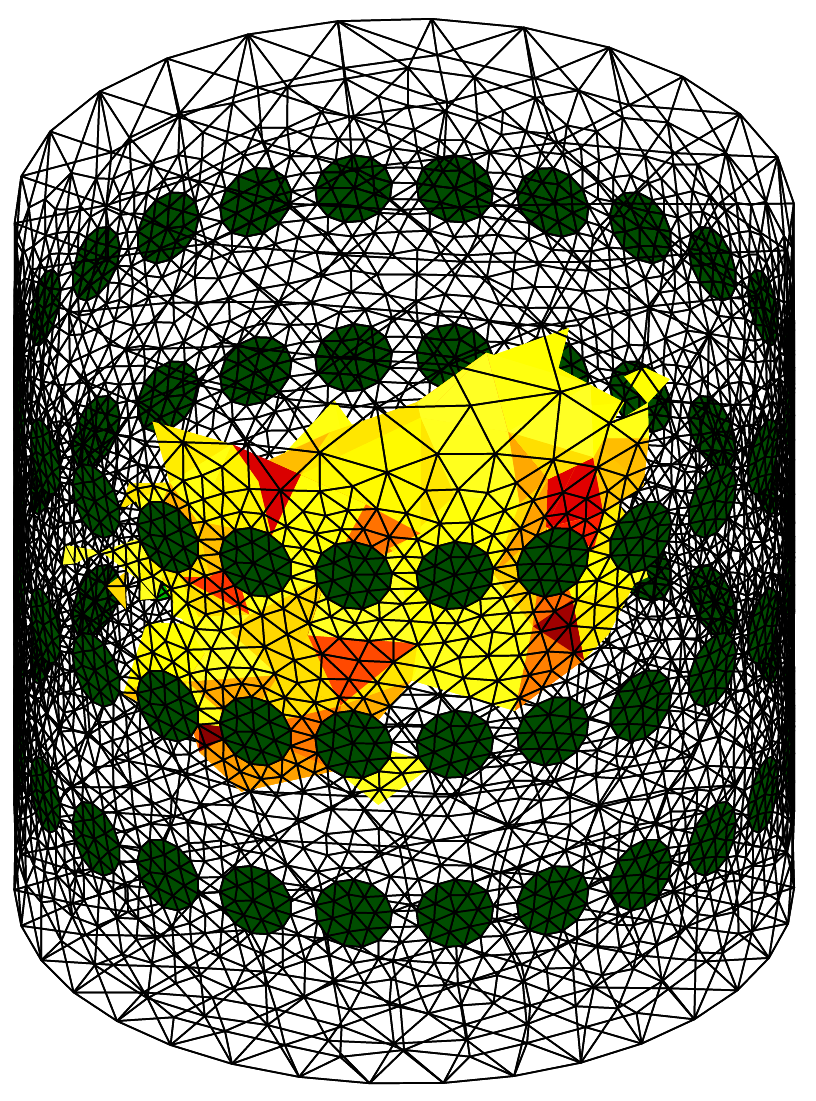}} 
\end{tabular}
\end{center}
\caption{Monotonicity-based reconstructions using measurements on $32$ electrodes including voltages on current driven electrodes (1st column) and after replacing them with geometrically interpolated values (with $r=0.7$, $r=0.8$ and $r=0.9$ in the 2nd, 3rd and 4th column),
and with linearly interpolated values (5th column). 
}
\label{fig:3D_reconstructions_beta}
\end{figure}


\section{Conclusion and discussion}\label{Sec:conlusion_and_discussion}

Rigorously justified reconstruction methods for EIT, such as monotonicity-based methods, require
voltage measurements on current-driven electrodes which are difficult to obtain in practice.
In this work, we have developed a method to interpolate the voltages on current-driven electrodes
from measurements on current-free electrodes for difference EIT settings. The interpolation 
is based on the geometry-specific smoothness of difference EIT data that 
arises from the fact that the difference potential solves an elliptic source problem.

Our method requires the a-priori knowledge of an upper bound of the support of the conductivity change.
The implementation of our new interpolation method is particularly simple and computationally cheap. It suffices to
identify the partition elements that belong to the upper bound of the support of the conductivity change, sum up the
corresponding columns of the standard sensitivity matrix, and to invert matrices that are formed from this sum. 
The method is well suited for real-time imaging. In a setting with $m$ electrodes and $r$ partition elements, the computational cost is merely of the order $O(r m^2)+O(m^3)$ which lowers to $O(m^2)$ if geometry-specific quantities can be precomputed. Our numerical experiments show that the method achieves a very high interpolation accuracy, and that a monotonicity-based reconstruction method performs well with interpolated data.

We have formulated the method for the interpolation of voltages on current-driven electrodes
in a real conductivity setting with adjacent-adjacent current driven patterns. The method
seems also applicable to interpolate other missing or erroneous data (e.g., bad electrode data in practical measurements),
and it should be extendable to complex conductivity (weighted frequency-difference) settings.

Our interpolation method is only applicable in difference EIT imaging and when the conductivity change occurs away from the boundary (which is arguably the most relevant case for recent commercial EIT applications). Voltage measurements in static EIT settings and voltage differences generated by a conductivity change at the boundary cannot be expected to possess the necessary smoothness properties to predict missing measurements by interpolation. Also, of course, interpolation is only necessary for EIT systems (such as the recent commercial ones) that do not provide reliable voltage measurements on current-driven electrodes. 

Finally, let us stress that voltages on current-driven electrodes are only
required by rigorously justified reconstruction methods such as the monotonicity method described herein. In practical applications, it is more common to use generic reconstruction algorithms that are based on minimizing a regularized data-fit functional, where missing measurements can simply be deleted from the residuum term.
There are no rigorous convergence results for such generic optimization-based approaches,
but they practically perform very well, and, so far, the stronger theoretical basis of rigorously justified methods has not lead to superior reconstructions in practice. 
Medical imaging algorithms must meet the highest standards, both
in their practical performance and their theoretical justification. 
Hence, it seems highly desirable to develop practically well-working algorithms for which at least certain properties can also be theoretically guaranteed.


\section*{References}

\bibliography{literaturliste}
\bibliographystyle{abbrv}

\end{document}